\documentclass[a4paper,12pt]{amsart}
\usepackage{amsmath}
\usepackage{amsthm}
\usepackage{amscd}
\usepackage{amssymb}
\usepackage{amsfonts}
\usepackage{latexsym}
\usepackage[mathscr]{eucal}

\newtheorem{thm}{Theorem}[section]
\newtheorem{prop}[thm]{Proposition}
\newtheorem{lemma}[thm]{Lemma}

\def\Ker{\mathop{\rm {Ker}}\nolimits}

\begin{document}
\title{Strongly irreducible operators and indecomposable  
representations of quivers on infinite-dimensional Hilbert spaces}
\author{Masatoshi Enomoto}
\address[Masatoshi Enomoto]{Institute of Education and Research, Koshien University, Takarazuka, Hyogo 665-0006, Japan}
\author{Yasuo Watatani}
\address[Yasuo Watatani]{Department of Mathematical Sciences, 
Kyushu University, Motooka, Fukuoka, 819-0395, Japan}
\maketitle

\begin{abstract}
We study several classes of indecomposable  representations of quivers on 
infinite-dimensional Hilbert spaces and their relation.
Many examples are constructed using 
strongly irreducible operators. Some problems in operator theory are 
rephrased in terms of representations of quivers.  
We shall show two kinds of constructions of quite non-trivial 
indecomposable Hilbert representations $(H,f)$ of 
the Kronecker quiver such that 
$End(H,f) ={\mathbb C} I$ which is called transitive. One is 
a perturbation of a 
weighted  shift operator by a rank-one operator. 
The other one is a modification of an unbounded operator 
used by  Harrison,Radjavi and Rosenthal to 
provide a transitive lattice.

\medskip


\medskip

\noindent AMS SUBJECT CLASSIFICATION: Primary 47A65, Secondary 
46C07,47A15, 15A21, 16G20, 16G60.
\end{abstract}

\section{Introduction}

A bounded  operator $T$ on a Hilbert space $H$ is said to be 
strongly irreducible if $T$ cannot be decomposed to a 
non-trivial (not necessarily orthogonal) direct sum of two operators, 
that is, if 
there exist no non-trivial invariant (closed) subspaces $M$ and $N$ 
of $T$ such that $M \cap N = 0$ and $M + N = H$. 
A strongly irreducible operator is an 
infinite-dimensional analog of a Jordan block. 
The notion of strongly irreducible operator was 
introduced by F. Gilfeather in \cite{Gi}. 
We refer to 
good monographs \cite{JW1} and  \cite{JW2} by 
Jiang and Wang on 
strongly irreducible operators. 

On the other hand Gabriel \cite{Ga} introduced a 
finite-dimensional (linear) representations of quivers by 
attaching vector spaces and linear maps for 
vertices and edges of quivers respectively.   
A finite-dimensional indecomposable representation 
of a quiver is a direct graph generalization of a Jordan block.   
We regard indecomposable representation of  a quiver on a 
Hilbert space as an infinite-dimensional generalization of 
both  a Jordan block and a finite-dimensional 
indecomposable representation of a quiver.  
We study several classes of indecomposable  representations of quivers on 
infinite-dimensional Hilbert spaces and their relation.
Many examples of indecomposable representations of quivers
are constructed using 
strongly irreducible operators. Moreover 
some problems in operator theory are 
rephrased in terms of representations of quivers.

Remember that we studied the relative positions of subspaces in a separable 
infinite-dimensional Hilbert space in \cite{EW1} after  Nazarova \cite{Na1}, 
 Gelfand and 
Ponomarev \cite{GP}. We shall describe a close relation  between 
the relative positions of subspaces in Hilbert spaces and 
Hilbert representations of quivers in the final section.

In our paper we only need the beginning of the theory of representations 
of quivers on finite-dimensional vector space, 
for example, see Bernstein-Gelfand-Ponomarev \cite{BGP}, 
Donovan-Freislish \cite{DF}, V. Dlab-Ringel \cite{DR}, 
Gabriel-Roiter \cite{GR}, 
Kac \cite{Ka}, Nazarova \cite{Na2} \dots .

We should remark that  locally  scalar representations of quivers 
in the category of Hilbert spaces were introduced by 
Kruglyak and Roiter \cite{KRo}. They associate 
operators and their adjoint operators with arrows 
and classify them up to the unitary 
equivalence.  They proved an 
analog of Gabriel's theorem.  Their study is connected 
with representations of *-algebras generated by 
linearly related orthogonal projections
, see for example, S. Kruglyak, V. Rabanovich and Y. Samoilenko
\cite{KRS} and Y. P. Moskaleva and Y. S. Samoilenko \cite{MS}.

In \cite{EW3}, we constructed an indecomposable infinite-dimensional 
Hilbert representation for any quiver whose underlying undirected 
graph is one of extended Dynkin diagrams 
$\tilde{A_n} \ (n \geq 0)$,
$ \tilde{D_n} \ (n \geq 4)$, 
$\tilde{E_6}$,$\tilde{E_7}$ and $\tilde{E_8}$,  
using the unilateral shift $S$. If we replace the unilateral shift $S$ 
there by any strongly irreducible operator, then the corresponding 
Hilbert representation is still indecomposable by the same 
calculation. This fact also suggests us to use strong irreducible 
operators to construct indecomposable Hilbert representations 
of quivers.

We recall infinite-dimensional representations in 
purely algebraic setting. In 
\cite{Au} Auslander  found that if a finite-
dimensional algebra 
is not of finite representation type,
then there exist indecomposable modules which are not of finite length.
Ringel \cite{Ri1} developed 
  a general theory of infinite-dimensional  
representations of tame, 
hereditary algebra. There exist many works after them and 
they form an active area of research in representation
theory of 
algebras.

In our paper we study infinite-dimensional 
Hilbert (space) representations 
 of quivers 
using operator theory.  We note that  
there exist subtle difference among purely algebraic infinite-dimensional 
representations of quivers, infinite-dimensional Banach (space) 
representations of quivers and infinite-dimensional Hilbert (space)
representations of quivers. We also note that the analytic aspect of 
Hardy space is quite important
 in our setting. For example, if 
the endomorphism algebra of a Hilbert representation of a quiver 
is isomorphic to  the Hardy algebra $H^{\infty}({\mathbb T})$, 
then the representation is indecomposable, because the 
the Hardy algebra $H^{\infty}({\mathbb T})$ has no non-trivial 
idempotents by the F. and M. Riesz Theorem. This is indeed the case of 
the Hilbert representation  corresponding to 
the unilateral shift operator. In this way 
we believe  that the analytic operator
 algebra theory will come  in here.

We shall show  two kinds of constructions of quite non-trivial 
indecomposable Hilbert representation $(H,f)$ of 
the Kronecker quiver such that 
$End(H,f) ={\mathbb C} I$ which is called transitive. One is 
a perturbation of a 
weighted  shift operator by a rank-one operator. 
This is an analogue of a construction of 
indecomposable representations using linear functionals 
on the space $K(x)$ of rational functions over 
an algebraically closed 
field $K$ studied in representation theory of algebras, for example, see
Ringel \cite{Ri2}, Fixmann \cite{Fi}, 
Okoh \cite{Ok}  and Dean- Zorzitto \cite{DZ}. 
We replace the rational function field $K(x)$ by Hardy spaces 
$H^{\infty}({\mathbb T})$ or $H^2({\mathbb T})$ properly in our setting . 
We have an analogy of ring extension  between 
$({\mathbb C}[x] \subset {\mathbb C}(x))$ and 
$({\mathbb C}[x] \subset H^{\infty}({\mathbb T}))$ in mind. 
Our analogy is supported by an important fact that 
both the rational function field ${\mathbb C}(x)$ and 
the Hardy algebra $H^{\infty}({\mathbb T})$ have no non-trivial idempotents.
But we warn the readers of subtle differences among them.

The other construction of transitive representations of 
the Kronecker quiver is given by a modification of an unbounded operator 
used by  Harrison,Radjavi and Rosenthal \cite{HRR} to 
provide a transitive lattice.

\section{Hilbert representations of quivers}

A quiver $\Gamma=(V,E,s,r)$ is a quadruple consisting of 
the set $V$ of vertices, the set $E$ of arrows, 
and two maps $s, r : E \rightarrow V$, which 
associate with each arrow $\alpha \in E$ its 
support $s(\alpha)$ and range  $r(\alpha)$. We 
sometimes denote by $\alpha : x \rightarrow y$ 
an arrow with $x = s(\alpha)$ and $y = r(\alpha)$. 
Thus a quiver is just a directed graph. We 
denote by $|\Gamma|$ 
the underlying undirected graph of a quiver $\Gamma$. 
A quiver $\Gamma$ is said to be connected if $|\Gamma|$ 
is a connected graph. A quiver $\Gamma$ is said to be finite 
if both $V$ and $E$ are finite sets.  
A path of length $m$ 
is a finite sequence $\alpha = (\alpha_1, \dots,\alpha_m)$ 
of arrows such that $r(\alpha_k) = s(\alpha_{k+1})$ for 
$k = 1,\dots, m-1$. Its support is $s(\alpha) = s(\alpha_1)$ 
and its range is $r(\alpha) = r(\alpha_m)$. A path of 
length $m \geq 1$ is called a {\it cycle } if its support 
and range coincide. A cycle of length one is called  a 
{\it loop}. A quiver is called {\it acyclic} if it contains 
no cycles. 

\bigskip

\noindent  
{\bf Definition.}
Let $\Gamma=(V,E,s,r)$ be a finite quiver. We say 
that $(H,f)$ is a  {\it Hilbert representation} of $\Gamma$ 
if $H=(H_{v})_{v\in V}$  is a family of  Hilbert spaces 
and $f=(f_{\alpha})_{\alpha \in E}$ is a family of
 bounded linear operators such that $f_{\alpha} : 
H_{s(\alpha)}\rightarrow H_{r(\alpha)}$ for $\alpha \in E$. 

\bigskip

\noindent  
{\bf Definition.} Let $\Gamma=(V,E,s,r)$ be a finite quiver. 
Let $(H,f)$ and $(K,g)$ be Hilbert representations of $\Gamma.$ 
A {\it homomorphism} $T : (H,f) \rightarrow (K,g)$  is a 
family $T =(T_{v})_{v\in V}$ of bounded operators 
$T_v : H_v \rightarrow K_v$ satisfying,  
for any arrow $\alpha \in E$ 
$$
T_{r(\alpha)}f_{\alpha}=g_{\alpha}T_{s(\alpha)}. 
$$
The composition $T \circ S$ of homomorphisms $T$ and $S$ is defined 
by the composition of operators:  
$(T \circ S)_v = T_v S_v$ for $v\in V$.  
Thus we have obtained 
a category $HRep (\Gamma)$ of Hilbert representations of $\Gamma$  

We denote by 
$Hom((H,f),(K,g))$ the set of homomorphisms 
$T : (H,f) \rightarrow (K,g)$. 
We denote by  $End (H,f): =Hom((H,f),(H,f))$  
the set of endomorphisms. Then we can regard 
$End (H,f)$ as a  subalgebra of $\oplus _{v \in V}B(H_v)$.

In the paper we carefully distinguish the 
following two classes of operators. 
A bounded operator $A$ on a Hilbert space is 
called an {\it idempotent}  if $A^2 = A$ and 
$A$ is said to be a {\it projection} if $A^2 = A = A^*$.   
 We denote by 
\begin{align*}
& Idem (H,f) : =\{T\in End (H,f) \ | \ T ^2 = T \} \\
&= \{T=(T_v)_{v \in V} \in End (H,f) \ | \ T_v ^2 = T_v \ 
({\text for \ any }\  v \in V) \}
\end{align*}
the set of idempotents of $End (H,f)$.  
Let $0 = (0_{v})_{v \in V}$ 
be a family of zero operators $0_{v}$ and 
$I = (I_{v})_{v\in V}$ 
be a family of identity operators $I_{v}$. 
The both endomorphisms $0$ and $I$ are in $Idem(H,f)$.

 Let $\Gamma=(V,E,s,r)$ be a finite quiver and 
$(H,f)$, $(K,g)$  be Hilbert representations of $\Gamma.$ 
We say that
$(H,f)$ and $(K,g)$ are {\it isomorphic}, denoted by  
$(H,f)\cong (K,g)$, 
if there exists an isomorphism $\varphi : (H,f) \rightarrow (K,g)$, 
that is, there exists a family  
$\varphi=(\varphi_{v})_{v\in V}$ of bounded invertible operators
$\varphi_{v}\in B(H_{v},K_{v})$ such that, for any arrow 
$\alpha \in E$, 
$$
\varphi_{r(\alpha)}f_{\alpha}=g_{\alpha
}\varphi_{s(\alpha)}.
$$
Hilbert representations $(H,f)$ and $(K,g)$ of $\Gamma$ 
are said to be 
{\it relatively prime} if 
$Hom((H,f),(K,g)) = 0$ and $Hom((K,g),(H,f)) =0 $. 
If two non-zero Hilbert representations $(H,f)$ and $(K,g)$ are  
relatively prime, then they are not isomorphic.

We say that $(H,f)$ is a 
finite-dimensional representation if $H_v$ is 
finite-dimensional for all $v \in V$. 
And $(H,f)$ is an  
infinite-dimensional representation if $H_v$ is 
infinite-dimensional for some $v \in V$.

We  shall recall  a notion of 
indecomposable representation in \cite{EW3},   
that is, a representation  which cannot be decomposed into a direct 
sum of smaller representations  anymore. 

\bigskip

\noindent  {\bf Definition.}(Direct sum) Let $\Gamma=(V,E,s,r)$ 
be a finite quiver. Let
$(K,g)$ and $(K',g')$ be Hilbert representations
of $\Gamma.$  Define the direct sum 
$(H,f)=(K,g)\oplus(K',g')$  by 
$$
H_{v}=K_{v}\oplus K'_{v}\ (\text{ for } v\in V) 
\ \text{ and } \ 
f_{\alpha}=g_{\alpha }\oplus g'_{\alpha } \ 
(\text{ for } \alpha \in E).
$$

We say that a Hilbert representation $(H,f)$ is zero, denoted by 
 $(H,f) =0 $,  if   
$H_v = 0$ for any $v \in V$. 

\bigskip

\noindent 
 {\bf Definition.}(Indecomposable representation)
A Hilbert representation $(H,f)$ of $\Gamma$ 
is called  {\it decomposable} if 
$(H,f)$ is isomorphic to a direct sum of two 
non-zero Hilbert representations.  
A non-zero Hilbert representation $(H,f)$ of $\Gamma$ 
is said to be  {\it indecomposable} if 
it is not decomposable, that is, 
if $(H,f)\cong(K,g) \oplus (K',g')$ 
then $(K,g) \cong 0$ or 
$(K',g') \cong 0$.

The following proposition is used frequently 
to show the indecomposability in concrete examples.

\begin{prop}\rm{(}\cite[Proposition 3.1.]{EW3}\rm{)} 
 Let $(H,f)$be a Hilbert representation of a quiver $\Gamma$.
Then the following conditions are equivalent: 
\begin{enumerate}
\item $(H,f)$ is indecomposable.
\item $ Idem(H,f) = \{0,I\}$. 
\end{enumerate}
\label{prop:indecomposable-idempotent}
\end{prop}

\bigskip
\noindent 
{\bf Remark.}
  $(H,f)$ is decomposable if and only if there exist 
families  $K = (K_x)_{x\in V}$ and 
$K' = (K'_x)_{x\in V}$ of  
closed subspaces $K_x$ and $K'_x$ of $H_x$ with 
$K_x \cap K'_x = 0$ and $K_x + K'_x = H_x$ 
satisfying $K$ and $K'$ are non-zero 
such that 
 $f_{\alpha} K_x \subset  K_y$  and 
$f_{\alpha} K'_x \subset  K'_y$ for any arrow  
$\alpha : x \rightarrow y $.

The indecomposability of Hilbert representations of 
a quiver is 
an isomorphic invariant,  
but it is {\it not} a unitarily equivalent 
invariant. Therefore we 
cannot replace 
the set  $Idem (H,f)$ of idempotents of endomorphisms 
by the subset of idempotents  of endomorphisms which are 
consists of projections to show the indecomposability. 

\bigskip

\noindent 
{\bf Example 1.} 
Let $H_0= \mathbb C ^2$. Fix an angle $\theta$
with $0 < \theta < \pi /2$.  Consider one-dimensional 
subspaces $H_1 = \mathbb C(1,0)$ and
$H_2 = \mathbb C(\cos \theta, \sin \theta)$ of $H_0$ 
spanned by vectors $(1,0)$ and $(\cos\theta, \sin\theta)$
in $H_0$. 
Consider the following quiver $\Gamma$ : 
\[
\circ_1 \overset{\alpha_1}{\longrightarrow}
 \circ_0 \overset{\alpha_2}\longleftarrow \circ_2
\]
Define a Hilbert 
representation $(H,f)$ of $\Gamma$ by $H = (H_i)_{i= 0,1,2}$ and 
canonical inclusion maps $f_i = f_{\alpha _i} : 
H_i \rightarrow H_0$ 
for $i = 1,2$. Then the Hilbert representation $(H,f)$ is 
decomposable.  But if  an idempotent $P=(P_v)_{v \in V} \in End (H,f)$ 
satisfies that  $P_v$ is a projection for any $v \in V$, 
then $P = 0$ or $P = I$. In fact, 
since $H_0 = H_1 + H_2$ and $H_1 \cap H_2 = 0$, 
for any $x \in H_0$, there exist unique $x_1 \in H_1$ and $x_2 \in H_2$ 
such that $x = x_1 + x_2$. There exists an idempotent 
$T_0: H_0 \rightarrow H_0$ such that $T_0x = x_1$. 
Put $T_1 = id: H_1 \rightarrow H_1$ and 
$T_2 = 0: H_2 \rightarrow H_2$. Then 
$T = (T_i)_{i= 0,1,2}$ is an idempotent in 
$End (H,f)$ such that  $T \not= 0$ and $T \not=I$. Hence 
$(H,f)$ is decomposable.  But take any idempotent
 $P=(P_v)_{v \in V} \in End (H,f)$ 
such that  $P_v$ is a projection for any $v \in V$. Then $P_0H_1 \subset H_1$ 
and $P_0H_2 \subset H_2$. Let $E_i$ be the projection of $H_0$  onto $H_i$ for 
$i = 1,2$. Since $P_0$ is self adjoint, 
 $P_0$ commutes with $E_1$ and $E_2$. Since the  $C^*$-algebra 
generated by $E_1$ and $E_2$ is $B(H_0) = M_2({\mathbb C})$ and 
$P_0$ is a projection, $P_0 = 0$ or $P_0 = I.$ If $P_0 = 0$, then 
$P_1= 0$ and $P_2 = 0$. Similarly if  $P_0 = I$, then 
$P_1= I$ and $P_2 = I$. Hence $P =0$ or $P = I$. 
We remark that  the system 
$(H_0;H_1,H_2)$ of two subspaces is isomorphic to 
\[ 
({\mathbb C}^2 ; {\mathbb C}\oplus 0, 0 \oplus {\mathbb C}) 
\cong (\mathbb C; \mathbb C, 0) \oplus (\mathbb C; 0, \mathbb C).
\]
Hence the relative position of two subspaces 
$(H_0;H_1,H_2)$ is decomposable. 
See \cite[Remark after Proposition 3.1.]{EW3}. 

\bigskip
\noindent  
{\bf Definition.}
A non-zero Hilbert representation $(H,f)$ of a quiver 
$\Gamma$ is called {\it irreducible} if 
 $P=(P_v)_{v \in V} \in End (H,f)$ is an idempotent and 
 $P_v$ is a projection for any $v \in V$, then $P = 0$ or 
$P = I$.  A non-zero Hilbert representation 
$(H,f)$ is not irreducible  if and only if there exist 
families  $K = (K_x)_{x\in V}$ and 
$K' = (K'_x)_{x\in V}$ of 
closed subspaces $K_x$ and $K'_x$ of $H_x$ with 
$K_x \perp  K'_x $ and $K_x + K'_x = H_x$ 
satisfying  $K$ and $K'$ are non-zero 
such that 
 $f_{\alpha} K_x \subset  K_y$  and 
$f_{\alpha} K'_x \subset  K'_y$ for any arrow  
$\alpha : x \rightarrow y $. 
For example, the Hilbert representation $(H,f)$ 
in Example 1 above is irreducible but is not indecomposable. 
We should be careful that  irreducibility 
is a unitarily invariant notion and not a 
isomorphically  invariant notion.

\bigskip

We recall the following elementary but fundamental relation 
between Hilbert representation theory of quivers and single operator theory:

\begin{thm}[\cite{EW3}]
Let $L_1$ be  one-loop quiver, so that $L_1$ has  one vertex $1$ and 
one arrow $\alpha : 1 \rightarrow 1$. The underlying 
undirected graph is an extended Dynkin 
diagram $\tilde{A_0}$.  For a bounded operator $A$ 
on a Hilbert space $H$, consider a 
Hilbert representation $(H^A,f^A)$ of $L_1$ such that $H^A_1 =H$ 
and $f^A_{\alpha} = A$. Then  the Hilbert representation 
$(H^A,f^A)$ is indecomposable 
if and only if $A$ is strongly irreducible. The endomorphism ring 
$End (H,f)$ can be identified with the commutant $\{A\}'$. 
Moreover two bounded operators $A$ and $B$ are 
similar if and only if the corresponding 
Hilbert representations $(H^A,f^A)$ and $(H^B,f^B)$are isomorphic. 
\end{thm}

Therefore  it is fruitful to regard 
the study of indecomposable Hilbert representations of general quivers 
as a generalization of the study of strongly irreducible operators. 
\bigskip

\noindent
{\bf Example 2.} Consider one-loop quiver $L_1$ as above. 
Let $H =  \ell^2(\mathbb N)$ and 
$S \in B(H)$ the  unilateral shift. Then 
$(H^S,f^S)$ is indecomposable. 

\bigskip
\noindent  
{\bf Definition.} Let $\Gamma=(V,E,s,r)$ be a finite quiver and 
$(H,f)$ a Hilbert representation of $\Gamma$. 
A Hilbert representation $(K,g)$ of $\Gamma$ is called a 
{\it subrepresentation} of $(H,f)$ if for any vertex $v \in V$, $K_v$ is a 
(closed) subspace of $H_v$ and for any edge $\alpha \in E$, 
$g_{\alpha} = f_{\alpha}|_{K_{s(\alpha)}}$. In particular 
we have  $f_{\alpha}(K_{s(\alpha)}) \subset K_{r(\alpha)}$. 

\bigskip
\noindent  
{\bf Definition.}
A non-zero Hilbert representation $(H,f)$ of a quiver 
$\Gamma$ is called {\it simple} 
if $(H,f)$ has only trivial subrepresentations $0$ and $(H,f)$. 
A Hilbert representation $(H,f)$ of $\Gamma$ is called 
{\it canonically simple } if 
there exists a vertex $v_{0}\in V$ such that $H_{v_{0}}=\mathbb{C}$,  
$H_{v}=0$ for any other vertex
$v\neq v_{0}$ and $f_{\alpha}=0$
for any $\alpha\in E$. It is clear that if 
a Hilbert representation $(H,f)$ of $\Gamma$ is 
canonically simple, then $(H,f)$ is simple. 
It is trivial that if 
a Hilbert representation $(H,f)$ of $\Gamma$ is 
simple, then $(H,f)$ is indecomposable.

We can rephrase the invariant subspace problem 
in terms of simple representations of a one-loop quiver. 
Let $L_1$ be  one-loop quiver, so that $L_1$ has  one vertex $1$ and 
one arrow $\alpha : 1 \rightarrow 1$.  Any bounded operator $A$ 
on a non-zero Hilbert space $H$ gives a 
Hilbert representation $(H^A,f^A)$ of $L_1$ such that $H^A_1 =H$ 
and $f^A_{\alpha} = A$.
Then the operator $A$ has only trivial invariant subspaces 
if and only if the Hilbert representation $(H^A,f^A)$ of $L_1$ is simple. 
If $H$ is one-dimensional and $A$ is a non-zero scalar operator,
then the Hilbert representation $(H^A,f^A)$ of $L_1$
is simple but is not canonically simple.
If $H$ is finite-dimensional with $\dim H \geq 2$, 
then the Hilbert representation $(H^A,f^A)$ of $L_1$ is not 
simple, because any operator $A$ on $H$ 
has a non-trivial invariant subspace. 
If $H$ is countably infinite-dimensional, then 
we do not know 
whether the Hilbert representation    
$(H^A,f^A)$ of $L_1$ is not 
simple. In fact this is 
the invariant subspace problem, that  
is,  the question whether any operator $A$ on $H$ 
has a non-trivial (closed)  invariant subspace.

\bigskip
\noindent  
{\bf Definition.}
A Hilbert representation $(H,f)$ of a quiver $\Gamma$ is called 
 {\it transitive} 
if $End(H,f) = {\mathbb C}I$. It is clear that 
if a Hilbert representation $(H,f)$ is canonically simple, then 
$(H,f)$ is transitive. 
If a Hilbert representation $(H,f)$ of $\Gamma$ is 
transitive , then $(H,f)$ is indecomposable. In fact, since 
$End(H,f) = {\mathbb C}I$, any 
idempotent endomorphism $T$ is $0$ or $I$. In purely algebraic 
setting, a representation of a quiver is called a {\it brick} if 
its endomorphism ring is a division ring. But for a 
Hilbert representation $(H,f)$ of a quiver, $End(H,f)$ is a 
Banach algebra and not isomorphic to its purely algebraic 
endomorphism ring in general, because we only consider 
bounded endomorphisms.  By Gelfand-Mazur theorem,  
any Banach algebra over 
${\mathbb C}$ which is a division ring must be isomorphic to 
${\mathbb C}$. Therefore the reader may  use 
"brick" instead of "transitive Hilbert representation" if 
he does not confuse the difference between purely algebraic 
endomorphism ring and $End(H,f)$. 

\bigskip
\noindent  
{\bf Remark.} A lattice ${\mathcal L}$ of subspaces 
of a Hilbert space $H$ containing $0$ and $H$ is  called a 
transitive lattice if 
$$
\{A \in B(H) \ | \ AM \subset M 
\text{ for any } M \in {\mathcal L}\} = {\mathbb C}I.
$$
See, for example, Radjavi-Rosenthal \cite[4.7.]{RR}.

Let ${\mathcal L} = \{0,M_1,M_2, \dots, M_n, H\}$
be a finite  lattice.  Consider a $n$ subspace quiver $R_n=(V,E,s,r)$, 
that is, $V = \{1,2,\dots,n,n+1 \}$ and 
$E = \{\alpha_k \ | \ k = 1, \dots, n \}$ with $s(\alpha_k) = k$ 
and $r(\alpha_k) = n+1$ for $k = 1, \dots, n$. 
Then there exists a Hilbert representation $(K,f)$ of $R_n$ such that 
$K_k = M_k$, $K_{n+1} = H$ and $f_{\alpha_k} : M_k \rightarrow H$ is an 
inclusion for  $k = 1, \dots, n$. Then the lattice ${\mathcal L}$ is 
transitive  if and only if the corresponding Hilbert representation 
$(H,f)$ is transitive. 
This fact guarantees the terminology "transitive" 
in the above.

\bigskip
\noindent 
{\bf Example 3.}
 Let $L_2$ be 2-loop, that is,  $L_2$ is a quiver  
with one vertex $\{v\}$ and two loops
 $\{\alpha,\beta\}$. 
Let $S$ be the unilateral shift on $\ell^2(\mathbb N)$.  
Define a Hilbert representation $(H,f)$ of
$L_2$ by $H_{v} =  \ell^2(\mathbb N)$ and 
$f_{\alpha}=S, f_{\beta}=S^{*}$.  
Then $(H,f)$ is simple and transitive 
but  $(H,f)$ is not canonically simple. In fact, since 
$End(H,f)$ is given by 
the commutant $\{S,S^*\}'$ and $\{S,S^*\}' = {\mathbb C}I$,  
$(H,f)$ is transitive. Any subrepresentation of $(H,f)$ is given 
by the common invariant subspaces of $S$ and $S^*$, which is  
$0$ or $\ell^2(\mathbb N)$. 
Hence any subrepresentation of $(H,f)$ is $0$ or $(H,f)$.

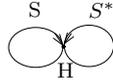
\unitlength=0.004in 
\begin{picture}(230,101)(0,-101)
\special{sh 1.0}%
\special{ar 880 396 8 8 -6.28318531 0}%
\special{ar 736 400 140 104 -6.28318530717959 0}%
\special{ar 1012 392 128 108 -6.28318530717959 0}%
\special{pa 836 328}%
\special{pa 876 388}%
\special{fp}%
\special{pa 828 352}%
\special{pa 876 388}%
\special{pa 860 332}%
\special{fp}%
\special{pa 920 328}%
\special{pa 884 380}%
\special{fp}%
\special{pa 900 324}%
\special{pa 884 380}%
\special{pa 932 344}%
\special{fp}%
\put(175,-59.8888888888889){{\tiny S}}%
\put(253,-60.8888888888889){{\tiny$S^{*}$}}%
\put(212,-138.888888888889){{\tiny H}}%
\end{picture}

\bigskip
\noindent 
{\bf Example 4.}
Let $K_3 =(V,E,s,r)$ be a 3-Kronecker quiver, so that 
$V = \{1,2\}$, $E = \{\alpha_1,  \alpha_2, \alpha_3 \}$ and 
$s(\alpha_i) = 1, r(\alpha_i) = 2$ for $i = 1,2,3$. 
Let $S$ be the unilateral shift on $\ell^2(\mathbb N)$.
Define a Hilbert representation $(H,f)$ by 
$H_1 = H_2= \ell^2(\mathbb N)$ and $f_1 = S$, 
$f_2 = S^*$, $f_3 = I$. 
Then the Hilbert representation $(H,f)$ is simple but  not 
canonically simple.

\bigskip
\noindent 
{\bf Example 5.}
A bounded operator$A$  on a Hilbert space $H$ is called {\it unicellular } 
if the lattice of  invariant subspaces of $A$ is totally ordered.
See \cite{RR} for unicellular operators. 
 For 
example, the unilateral shift $S$ is unicellular. 
Any non-zero unicellular 
operator is strongly irreducible. 
Let $L_1$ be  one-loop quiver, so that $L_1$ has  one vertex $1$ and 
one arrow $\alpha : 1 \rightarrow 1$. 
Consider a 
Hilbert representation $(H^A,f^A)$ of $L_1$ such that $H^A_1 =H$ 
and $f^A_{\alpha} = A$. If $A$ is unicellular, then 
Hilbert representation $(H^A,f^A)$ of $L_1$ is indecomposable. 
 Let $L_2$ be 2-loop, that is,  $L_2$ is a quiver  
with one vertex $\{v\}$ and two loops
 $\{\alpha,\beta\}$.  
Define a Hilbert representation $(H,f)$ of
$L_2$ by $H_{v} =  H$ and 
$f_{\alpha}=A, f_{\beta}=A^{*}$.  
If $A$ is a unicellular operator, then $(H,f)$ is simple. In fact any 
subrepresentation $(K,g)$ is 
given by a common invariant subspace $M$  of $A$ and $A^*$. Then 
$M$ and $M^{\perp}$ is an invariant subspace of $A$. 
Since $A$ is unicellular, $M =0$ or $M=H$.

We summarize   relations among several  classes  of Hilbert
representations of quivers. 

%%%%%%%
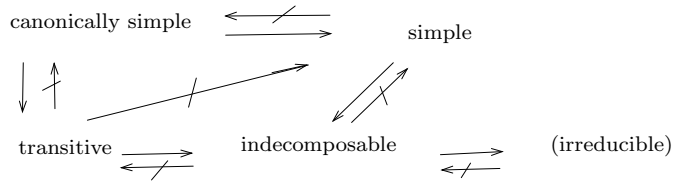
\unitlength=0.004in 
\begin{picture}(740,293)(0,-293)
\put(103,-94.8888888888889){{\tiny canonically simple}}%
\special{pa 1528 416}%
\special{pa 2072 408}%
\special{fp}%
\special{pa 2016 428}%
\special{pa 2072 408}%
\special{pa 2016 392}%
\special{fp}%
\put(619,-108.888888888889){{\tiny simple
}}%
\put(114,-258.888888888889){{\tiny transitive}}%
\put(402,-254.888888888889){{\tiny indecomposable}}%
\put(803,-254.888888888889){{\tiny (irreducible)}}%
\special{pa 2640 1040}%
\special{pa 2960 1024}%
\special{fp}%
\special{pa 2904 1044}%
\special{pa 2960 1024}%
\special{pa 2904 1008}%
\special{fp}%
\special{pa 992 1040}%
\special{pa 1352 1032}%
\special{fp}%
\special{pa 1296 1052}%
\special{pa 1352 1032}%
\special{pa 1296 1016}%
\special{fp}%
\special{pa 1360 1100}%
\special{pa 988 1108}%
\special{fp}%
\special{pa 1044 1088}%
\special{pa 988 1108}%
\special{pa 1044 1124}%
\special{fp}%
\special{pa 1228 1072}%
\special{pa 1144 1172}%
\special{fp}%
\special{pa 2948 1112}%
\special{pa 2648 1120}%
\special{fp}%
\special{pa 2704 1100}%
\special{pa 2648 1120}%
\special{pa 2704 1136}%
\special{fp}%
\special{pa 2796 1084}%
\special{pa 2748 1160}%
\special{fp}%
\special{pa 2396 560}%
\special{pa 2084 864}%
\special{fp}%
\special{pa 2112 812}%
\special{pa 2084 864}%
\special{pa 2136 836}%
\special{fp}%
\special{pa 2180 872}%
\special{pa 2468 592}%
\special{fp}%
\special{pa 2440 644}%
\special{pa 2468 592}%
\special{pa 2416 620}%
\special{fp}%
\special{pa 2328 684}%
\special{pa 2364 780}%
\special{fp}%
\special{pa 2068 312}%
\special{pa 1528 316}%
\special{fp}%
\special{pa 1584 296}%
\special{pa 1528 316}%
\special{pa 1584 336}%
\special{fp}%
\special{pa 1888 264}%
\special{pa 1772 352}%
\special{fp}%
\special{pa 480 560}%
\special{pa 476 816}%
\special{fp}%
\special{pa 460 760}%
\special{pa 476 816}%
\special{pa 496 760}%
\special{fp}%
\special{pa 644 800}%
\special{pa 640 564}%
\special{fp}%
\special{pa 660 620}%
\special{pa 640 564}%
\special{pa 624 620}%
\special{fp}%
\special{pa 672 656}%
\special{pa 580 704}%
\special{fp}%
\special{pa 812 860}%
\special{pa 1956 572}%
\special{fp}%
\special{pa 1768 612}%
\special{pa 1948 572}%
\special{fp}%
\special{pa 1896 604}%
\special{pa 1948 572}%
\special{pa 1888 568}%
\special{fp}%
\special{pa 1376 640}%
\special{pa 1336 792}%
\special{fp}%
\end{picture}

\noindent
{\bf Example 6.} Consider 2-loop quiver $L_2$, 
that is,  $L_2$ is a quiver  
with one vertex $\{v\}$ and two loops $\{\alpha,\beta\}$. 
Define $H_v = {\mathbb C}^2$ and 
$$
f_{\alpha} = A := 
\begin{pmatrix}
1 & 0 \\ 
0 & 0 \\
\end{pmatrix}, 
\ \ \ 
f_{\beta} = B :=
\begin{pmatrix}

0 & 1 \\ 
0 & 0 \\ 
\end{pmatrix}.
$$
Then the Hilbert representation $(H,f)$ is transitive but is not 
simple. Since 
$$
End(H,f) = \{T \in M_2({\mathbb C}) \ | \ TA = AT \text{ and } 
TB = BT \} = {\mathbb C}I, 
$$
$(H,f)$ is transitive. Define $K_v = {\mathbb C} \oplus 0$ and 
$g_{\alpha} = id_{K_v}$ and $g_{\beta} = 0$. Then $(K,g)$ is 
a subrepresentation such that $(K,g) \not= 0$ and  
$(K,g)\not= (H,f)$. Therefore $(H,f)$ is not simple. 

\bigskip
\noindent
{\bf Example 7.} Consider 2-loop quiver $L_2$, 
that is,  $L_2$ is a quiver  
with one vertex $\{v\}$ and two loops $\{\alpha,\beta\}$. 
Define $H_v = {\mathbb C}^2$ and 
$$
f_{\alpha} = A := 
\begin{pmatrix}
1 & 0 \\ 
0 & 0 \\
\end{pmatrix}, 
\ \ \ 
f_{\beta} = C :=
\begin{pmatrix}

1 & 1 \\ 
1 & 1 \\ 
\end{pmatrix}.
$$
Then the Hilbert representation $(H,f)$ is transitive and  
simple. In fact since 
$$
End(H,f) = \{T \in M_2({\mathbb C}) \ | \ TA = AT \text{ and } 
TC = CT \} = {\mathbb C}I, 
$$
$(H,f)$ is transitive. Let $(K,g)$ be subrepresentation of $(H,f)$. 
Since $K_v$ is a common invariant subspace for $A$ and $C$, 
$K_v$ is $0$ or $H_v = {\mathbb C}^2$. Hence $(K,g) = 0$ or 
$(K,g) = (H,f)$. Thus $(H,v)$ is simple.  

We collect some elementary facts: 

\begin{prop}
Let $(H,f)$ be a finite-dimensional Hilbert
representation of a quiver $\Gamma$. 
If $(H,f)$ is simple, then $(H,f)$ is transitive.
\end{prop}
\begin{proof}
Assume that $(H,f)$ is not transitive. 
Then the following two cases occur:\\
(A) There exists a vertex $u\in V$ and an endomorphism $T\in End(H,f)$ such
that $T_{u} \notin {\mathbb C}I_{u}$ .\\
(B) There exist vertices $v_{1}\ne v_{2}$ and 
scalars $\lambda_{1}\ne \lambda_{2}$ such that 
 $T_{v_{1}}=\lambda_{1}I_{v_{1}}$, $T_{v_{2}}=\lambda_{2}I_{v_{2}}$ 
with $H_{v_1} \not= 0$ and $H_{v_2} \not= 0$. \\
In either case we shall show that there exists a non-trivial 
subrepresentation. \\
The case (A): There exists a scalar $\lambda$ such that 
the eigenspace $0 \not= \Ker (T_{u}-\lambda I)\not= H_u$. 
For any vertex $v$ define $K_{v}=\Ker(T_{v}-\lambda I)$ . 
Then $f_{\alpha}(K_{s(\alpha)})\subset (K_{r(\alpha)})$, 
for any edge $\alpha \in E$. In fact, for $x \in K_{s(\alpha)}$, 
$$
T_{r(\alpha)}f_{\alpha}x = f_{\alpha}T_{s(\alpha)}x 
= f_{\alpha}\lambda x = \lambda f_{\alpha}x. 
$$
Let $g_{\alpha}$ be the restriction of $f_{\alpha}$ to $K_{s(\alpha)}$. 
Then $(K,g)$ is a non-trivial subrepresentation of $(H,f)$. \\
The case (B): Define a subrepresentation $(K,g)$ by 
$K_v = \Ker(T_{v}-\lambda_1 I)$ and 
$g_{\alpha}= f_{\alpha}|_{K_{s(\alpha)}}$. Since $K_{v_1} = H_{v_1}$ 
and $K_{v_2} = 0$. 
$(K,g)$ is a non-trivial subrepresentation of $(H,f)$.
\end{proof}

\begin{prop} Let $\Gamma$ be a finite quiver with no oriented cycles and 
$(H,f)$ a non-zero Hilbert representation of $\Gamma$. 
Then $(H,f)$ is simple if and only if $(H,f)$ is canonically simple. 
\end{prop}
\begin{proof}
It is trivial that if $(H,f)$ is canonically simple, then it is simple. 
Conversely assume that $(H,f)$ is simple. 
Since  $\Gamma$ is a quiver with no oriented cycles, there is a
sink  $v_{0}$ in $V$. First consider the case that $H_{v_0} \not= 0$. 
Choose a non-zero vector $x \in H_{v_0}$. 
Define a representation $(K,g)$ of $\Gamma$  
by  $K_{v}=\mathbb{C}x$ if $v={v_{0}}$ and 
$K_{v}=0$ if  $v \not= {v_{0}}$. Put  $g_{\alpha}=0$ 
for any ${\alpha} \in  E$. 
Since $v_{0}$ is a sink,  $(K,g)$ is a non-zero subrepresentation of $(H,f)$. 
Since $(H,f)$ is simple, $(H,f) = (K,g)$ . This means that $(H,f)$ is 
canonically simple.
Next consider the general case. Since 
$\Gamma$ is a finite quiver with no oriented cycles, 
there exists a vertex $v_1$ in $V$ such that $H_{v_1} \not= 0$ and 
$f_{\alpha}=0$ for any edge $\alpha \in E$ with $s(\alpha) = v_1$. 
Choose a non-zero vector $y \in H_{v_1}$.
Define a representation $(L,h)$ of $\Gamma$  
by  $L_{v}=\mathbb{C}y$ if $v={v_{1}}$ and 
$L_{v}=0$ if  $v \not= {v_{1}}$. Put  $h_{\alpha}=0$ 
for any ${\alpha} \in  E$. 
Then  $(L,h)$ is a non-zero subrepresentation of $(H,f)$. 
Since $(H,f)$ is simple, $(H,f) = (L,h)$ .
Therefore $(H,f)$ is canonically simple.
\end{proof}

\begin{prop}
Let $L_2$ be 2-loop, that is,  $L_2$ is a quiver  
with one vertex $\{v\}$ and two loops
 $\{\alpha,\beta\}$. 
Let $T$ be a bounded operator on an infinite dimensional 
Hilbert space $H$. Let $(H,f)$ be an infinite dimensional 
representation of $\Gamma$  such that
$H_{v} =H$ and $f_{\alpha}=T, f_{\beta}=T^{*}$. 
Then the following conditions are equivalent: 
\begin{enumerate}
\item $(H,f)$ is transitive . 
\item $(H,f)$ is simple.
\item $T$ is irreducible, that is, 
the commutant  $\{T,T^{*}\}^{\prime}= {\mathbb C}$.
\end{enumerate}
\end{prop}

\begin{proof}
Note that $End(H,f) = \{ A \in B(H) \ | \ AT=TA,\ AT^*=T^*A\}$. 
Any subrepresentation is given by a subspace $M$ of $H$ such that 
$TM \subset M$ and $T^*M \subset M$. Let $P$ be the projection of $H$ 
onto $M$. Then this  means that $P$ commutes with $T$ and $T^*$. 
Therefore these three condition (1), (2) and (3) are equivalent.
\end{proof}

\section{Hilbert representations of the Kronecker quiver}

The Kronecker quiver $K$ is a quiver with two vertices $\{1,2\}$ and 
two paralleled arrows $\{\alpha, \beta\}$:
$$
K : 1 ^{\overset {\alpha}{\longrightarrow}}
_{\underset {\beta}{\longrightarrow}} 2
$$
A Hilbert representation $(H,f)$ of the Kronecker quiver is given 
by two Hilbert spaces $H_1$, $H_2$ and two bounded operators 
$f_{\alpha}, f_{\beta}: H_1 \rightarrow H_2$. 

The finite-dimensional indecomposable representations of 
the Kronecker quiver $K$ was  partially classified by Weierstrass 
and completed by Kronecker. 
Any finite-dimensional indecomposable representation is one of 
the following up to isomorphism: 
\begin{enumerate}
\item $H_1 = {\mathbb C}^n, H_2 = {\mathbb C}^n, 
f_{\alpha} = \lambda I_n + J_n, \text { (Jordan block) }, 
 \lambda \in {\mathbb C}, f_{\beta}= I_n , \ \ n \geq 1$.  
\item $H_1 = {\mathbb C}^n, H_2 = {\mathbb C}^n, 
f_{\alpha} = I_n, f_{\beta} = \lambda I_n + J_n, \text { (Jordan block) }, 
\lambda \in {\mathbb C},  \ \ n \geq 1$.
\item $H_1 = {\mathbb C}^{n+1}, H_2 = {\mathbb C}^{n}, 
f_{\alpha} = [I_n, 0], f_{\beta}= [0,I_n], 
n \geq 0$.  
\item $H_1 = {\mathbb C}^{n}, H_2 = {\mathbb C}^{n+1}, 
f_{\alpha} = 
\begin{bmatrix}
I_n \\
0
\end{bmatrix} 
, f_{\beta}=  
\begin{bmatrix}
0 \\
I_n
\end{bmatrix}, 
n \geq 0.
$ 
\end{enumerate}
Moreover the case (2) is reduced to (1) if $\lambda \not=0$. 
Among these cases, any representation $(H,f)$ in (3),(4), 
n=1 of (1) or n = 1 of (2) 
is transitive and $End (H,f) = {\mathbb C}I$. Therefore it is interesting to 
find an infinite-dimensional indecomposable Hilbert representation $(H,f)$
of the Kronecker quiver $K$ and one with $End (H,f) = {\mathbb C}I$. 

\bigskip
\noindent
{\bf Example 8}
Let $K$ be the Kronecker quiver.
Let $S$ be the unilateral shift on $\ell^2(\mathbb N)$.
Define a Hilbert representation $(H,f)$ of $K$ by 
$H_{1} = H_2=  \ell^2(\mathbb N)$ and 
$f_{\alpha}=I, f_{\beta}=\lambda I + S, 
\ \lambda \in {\mathbb C}$. Then  
the Hilbert representation $(H,f)$ of $K$ is 
indecomposable and is not transitive.

Similarly, define a Hilbert representation $(L,g)$ of $K$ by 
$L_{1} = L_2=  \ell^2(\mathbb N)$ and 
$g_{\alpha}=I, g_{\beta}=\lambda I + S^* , 
\ \lambda \in {\mathbb C} $. Then  
the Hilbert representation $(L,g)$ of $K$ is
indecomposable and is not transitive. 
These are infinite-dimensional analog of 
the case (1) and (2) of the finite-dimensional indecomposable 
representation of $K$. 

\bigskip
We can replace the unilateral shift by a strongly 
irreducible operators in general. The 
following proposition shows that strongly irreducible 
operators are important to study  Hilbert representations of 
quivers. 

\begin{prop}
Let $K$ be the Kronecker quiver.
Let $A$ be a bounded operator on a 
Hilbert space ${\mathcal H}$. Let 
$(H,f)$ be a Hilbert representation 
defined by one of the following forms:
\begin{enumerate}
\item $H_1 = {\mathcal H}, H_2 = {\mathcal H}, 
f_{\alpha} = A, f_{\beta}= I$. 
\item $H_1 = {\mathcal H}, H_2 = {\mathcal H},
f_{\alpha} = I, f_{\beta} = A$. , 
\end{enumerate}
If $A$ is strongly irreducible, 
then the Hilbert representation $(H,f)$ of 
the Kronecker quiver $K$ is indecomposable. 
The representation $(H,f)$ is not transitive 
unless $\dim {\mathcal H} = 1$. 
Conversely if the Hilbert representation $(H,f)$ 
is indecomposable, then the operator $A$ is 
strongly irreducible. 
\end{prop}
\begin{proof}Assume that $A$ is strongly irreducible. 
Let $T = (T_1,T_2)$ be in $Idem (H,f)$. Then  
$T_2I = I T_1$ and $T_2A = AT_1$ in either case. 
Thus $T_1A = AT_1$ and $T_1$ is an idempotent. 
Since $A$ is strongly irreducible, $T_1 = 0$ or $T_1 = I$. 
Hence $T = 0$ or $T = I$, because $T_1 = T_2$. Therefore 
$(H,f)$ is indecomposable. 
Furthermore suppose that $\dim {\mathcal H} \not= 1$.
Since $A$ is strongly irreducible, $A$ is not a scalar operator. 
Then $(A,A)$ is in $End (H,f)$  and is not a scalar. 
Therefore $(H,f)$ is not transitive. 
Conversely, assume that $(H,f)$ is indecomposable. 
Let $ Q  \in {\mathcal H}$ be an idempotent 
operator such that  $QA = AQ$. 
Then  $T = (Q,Q)$ is in  $End (H,f)$. Since 
$(H,f)$ is indecomposable, $Q = 0 $ or $Q = I$. Hence 
$A$ is strongly irreducible.  

\end{proof}

\begin{lemma}
Let $K$ be the Kronecker quiver.
Let $A$  and $B$ be bounded operators on a 
Hilbert space ${\mathcal H}$. Let 
$(H,f)$ be a Hilbert representation 
defined by $H_1 = {\mathcal H}, H_2 = {\mathcal H}, 
f_{\alpha} = A$ and $f_{\beta}= B$. 
If $A$ is invertible, then 
there exists a bounded operator $C$ on ${\mathcal H}$, 
such that $(H,f)$ is 
isomorphic to a Hilbert representation $(L,g)$ 
with  $L_1 = {\mathcal H}, L_2 = {\mathcal H}, 
g_{\alpha} = I$ and $g_{\beta}= C$. 
Conversely 
if $(H,f)$ is isomorphic to a Hilbert representation $(L,g)$ 
with $g_{\alpha} = I$ and $g_{\beta}= C$, then $A$ is invertible. 
\end{lemma}
\begin{proof} Put $C = A^{-1}B$. Then 
$T = (T_1,T_2) := (A,I)$ gives a desired  isomorphism 
of $(H,f)$ onto $(L,g)$. The converse is clear. 
\end{proof}

\begin{prop}
Let $K$ be the Kronecker quiver.
Let $A$  and $B$ be bounded operators on a 
Hilbert space ${\mathcal H}$. Let 
$(H,f)$ be a Hilbert representation 
defined by $H_1 = {\mathcal H}, H_2 = {\mathcal H}, 
f_{\alpha} = A$ and $f_{\beta}= B$. 
If $xA+yB$ is invertible for some scalars $x$ and $y\not=0$, then 
there exists a bounded operator $T$ on ${\mathcal H}$,  
scalars $\lambda _0 \not= 0$ and $\lambda _1$
such that $(H,f)$ is 
isomorphic to a Hilbert representation $(L,g)$ 
with  $L_1 = {\mathcal H}, L_2 = {\mathcal H}, 
g_{\alpha} = T$ and $g_{\beta}= \lambda _0 I + \lambda _1T $. 
\end{prop}
\begin{proof} Put $T = (xA + yB)^{-1}A$. Then 
$$
(xA + yB)^{-1}B = (xA + yB)^{-1}(\frac{1}{y}(xA + yB)-\frac{x}{y}A) 
= \frac{1}{y}I - \frac{x}{y}T. 
$$
Therefore 
$T = (T_1,T_2) := (I,(xA + yB)^{-1})$ gives a desired  isomorphism 
of $(H,f)$ onto $(L,g)$. 
\end{proof}

We shall show that strongly irreducible operators are useful to 
construct indecomposable 
Hilbert representations of $n$-Kronecker quivers. 

\begin{thm}
Let $K_{n+1}$ be the $(n+1)$-Kronecker quiver, that is, 
$K_{n+1}$ is a quiver with two vertex $\{1, 2 \}$ and 
$n+1$ edges $\{\alpha_0, \dots, \alpha_{n} \}$ such that 
$s(\alpha_k) = 1$ and $r(\alpha_k) = 2$ for $k = 0,\dots, n$. 
Let $T$ be a bounded operator on a Hilbert space H. 
Define a Hilbert representation $(M,f) = (M_T,f_T)$ of $K_{n+1}$ by 
$M_1 = M_2 = H$ and 
$f_{\alpha_0} = \sum_{k=0}^n \lambda_kT^k$ for some 
scalars $\lambda_0 \not= 0, \lambda_1,\dots, \lambda_n$ and     
$f_{\alpha_k} = T^k$ for $k = 1,\dots,n$. Then 
the Hilbert representation $(M,f)$ is indecomposable if 
and only if $T$ is strongly irreducible. 
Moreover let $S$ be another bounded operator. 
Then the corresponding Hilbert representations 
$(M_T,f_T)$ and $(M_S,f_S)$ are isomorphic if and only if 
$T$ and $S$ are similar. 

\end{thm} 
\begin{proof}
We shall show that 
$$
End(M,f) = \{ (A,A) \in B(H)^2 \ | \ AT = TA \}
$$
In fact, let $(A,B)  \in B(H)^2$ be in $End(M,f)$. 
Then $BT^k = T^kA$ for  $k = 1,\dots,n$ and 
 $ B(\sum_{k=0}^n \lambda_kT^k) = ( \sum_{k=0}^n \lambda_kT^k)A$. 
Hence $\lambda_0 A =  \lambda_0 B$. Since 
$\lambda_0 \not= 0$, we have $A = B$. Then $AT = TA$. The converse is 
clear. Therefore 
$$
Idem (M,f) = \{ (A,A) \in B(H)^2 \ | \ AT = TA,\ \  A^2 = A  \} = \{0,I\}
$$ 
if and only if $T$ is strongly irreducible. The latter half 
of the statement  is similarly proved. 
\end{proof}

\begin{thm}
Let $K$ be the Kronecker quiver. 
Then there exist infinitely many 
infinite-dimensional indecomposable 
Hilbert representations of $K$ 
which are relatively prime each other.  
\end{thm} 
\begin{proof}
Let $S$ be the unilateral shift on $\ell^2(\mathbb N)$. 
For each $\lambda \in {\mathbb C}$ 
define an indecomposable Hilbert representation 
$(H^{\lambda},f^{\lambda})$ of $K$ by 
$H^{\lambda}_{1} = H^{\lambda}_2=  \ell^2(\mathbb N)$ and 
$f^{\lambda}_{\alpha}=I, f^{\lambda}_{\beta}=\lambda I + S$. 
For $\lambda,\mu  \in {\mathbb C}$, if 
$|\lambda - \mu| > 2$, then we shall show that 
$(H^{\lambda},f^{\lambda})$ and 
$(H^{\mu},f^{\mu})$ are relatively prime each other. 
On the contrary, assume that 
$Hom((H^{\lambda},f^{\lambda}),(H^{\mu},f^{\mu} )) \not= 0$. 
Then there exists a non-zero 
$T = (T_1, T_2) \in Hom((H^{\lambda},f^{\lambda}),(H^{\mu},f^{\mu} ))$. 
Then 
$$
T_2 I = I T_1, \  \ \ T_2(\lambda I + S) = ( \mu I + S)T_1.  
$$
Hence $T_1 = T_2$ and $(\lambda - \mu)T_1 = ST_1 -T_1S$. 
Then
$$
 2\|T_1\| < \|(\lambda - \mu)T_1\| = \|ST_1 -T_1S\| \leq 2\|T_1\|. 
$$
This is a contradiction. 
Hence $Hom((H^{\lambda},f^{\lambda}),(H^{\mu},f^{\mu} )) = 0$. 
Similarly we have 
$Hom((H^{\mu},f^{\mu} ), (H^{\lambda},f^{\lambda})) = 0$. 
\end{proof}

Since relatively prime Hilbert representations 
are not isomorphic, it is clear that 
there exist infinitely  many non-isomorphic 
infinite-dimensional indecomposable 
Hilbert representations of $K$. 
We can easily say more.

\begin{prop}
Let $K$ be the Kronecker quiver. 
Then there exist uncountably  many non-isomorphic 
infinite-dimensional 
indecomposable Hilbert representations $(H,f)$ of $K$.
\end{prop} 
\begin{proof}
Since $\lambda I + S$ and $\mu I + S$ are not similar 
for $\lambda \not= \mu \in {\mathbb C}$,  
the Hilbert representations 
$(H^{\lambda},f^{\lambda})$ and 
$(H^{\mu},f^{\mu})$ of the Kronecker quiver $K$ are 
not isomorphic. 
\end{proof}

\bigskip
\noindent
{\bf Example 9}
Let $K$ be the Kronecker quiver.
Let $S$ be the unilateral shift on $\ell^2(\mathbb N)$. 
Define a Hilbert representation $(H,f)$ of $K$ by 
$H_{1} = H_2=  \ell^2(\mathbb N)$ and 
$f_{\alpha}=S, f_{\beta}=S^{*}$. Then  
the Hilbert representation $(H,f)$ of $K$ is 
not indecomposable, so that $(H,f)$ of $K$ is not simple. In fact, 
let $\{e_n \ | \ n \in \mathbb N \}$ be a canonical basis 
of $\ell^2(\mathbb N)$. Define 
$L_1 = \{e_n \ | \ n \in \mathbb N  \text{ is odd } \}$ and 
$L_2 = \{e_n \ | \ n \in \mathbb N  \text{ is even } \}$. 
Consider the restrictions $g_{\alpha}=S|_{L_1}$ and 
$g_{\beta}=S^*|_{L_1}$. Then $(L,g)$ is a 
non-trivial subrepresentation of $(H,f)$. 
Similarly define 
$M_1 = \{e_n \ | \ n \in \mathbb N  \text{ is even } \}$ and 
$M_2 = \{e_n \ | \ n \in \mathbb N  \text{ is odd } \}$. 
Consider the restrictions $h_{\alpha}=S|_{M_1}$ and 
$h_{\beta}=S^*|_{M_1}$. Then $(M,h)$ is also a 
non-trivial subrepresentation of $(H,f)$ and 
$(H,f) = (L,g) \oplus (M,h)$. 

In the following we shall construct 
infinite-dimensional indecomposable Hilbert representations 
of the Kronecker quiver which are transitive. 

A perturbation of a weighted shift  by a rank-one operator is 
crucially used to show being transitive.

Let $H$ be a Hilbert space. Recall that for vectors $a,b \in H$, 
a rank-one operator $\theta_{a,b}$ is defined by 
$\theta_{a,b}(x) = (x|b)a, \ \ x \in H$.  

\begin{thm} \label{thm:perturbation}
Let $K=(V,E,s,r)$ be the Kronecker quiver
so that  $V=\{1,2\}$, $E=\{\alpha,\beta \}$ with 
 $s(\alpha)=1,s(\beta)=1,r(\alpha)=2$ and $r(\beta)=2,$. 
Let $S$ be the  unilateral shift on 
${\mathcal H}=\ell^{2}(\mathbb{N})$ with a canonical 
basis $\{e_{1},e_{2},...\}$. 
For a bounded weight vector 
${\lambda}=(\lambda_{1},\lambda_{2},...)\in \ell^{\infty}$ 
we associate with a diagonal operator 
$D_{\lambda}=\text{diag}(\lambda_{1},\lambda_{2},...)$, 
so that $SD_{\lambda}$ is a weighted shift operator. 
We assume that $\lambda_{i}\ne\lambda_{j}$ if $i\ne j$. 
Take a vector  
$\overline{w}=(\overline{w_{n}})_n \in \ell^{2}(\mathbb{N})$ such 
that $w_{n}\ne 0$ for any $n\in \mathbb{N}$.
Let $(H,f)$ be a Hilbert representation such that  
 $H_{1}=H_{2}={\mathcal H}=\ell^{2}(\mathbb{N})$ and
$f_{\alpha}=S,f_{\beta}= T := SD_{\lambda}+\theta_{e_{1},\overline{w}}$.
i.e., $f_{\beta}$ is a perturbation of 
a weighted shift  by a rank-one operator. 
Then the Hilbert representation $(H,f)$ is transitive.
\end{thm}
\begin{proof}
We need an infinite matrix presentation of $T$: 
$$
T=SD_{\lambda}+\theta_{e_{1},\overline{w}}= 
\left(
\begin{array}{cccc}
w_{1}&w_{2}&w_{3}&... \\
\lambda_{1}&0&0&... \\
0&\lambda_{2}&0&... \\
0&0&\lambda_{3}&...\\
0&0&0& ...\\
\end{array}
\right)
$$
Take any  $(\phi,\psi)$ in $End(H,f)$. Write $\phi = (c_{ij})_{ij} 
\in B(\ell^{2}(\mathbb{N}))$ and 
$$
\psi=\left(
\begin{array}{cccc}
\alpha&\beta_{1}&\beta_{2}&\cdots \\ 
\gamma_{1}&d_{11}&d_{12}&\cdots\\
\gamma_{2}&d_{21}&d_{22}&\cdots\\
\vdots&\ddots&\ddots&\ddots
\end{array}\right)
\in B(\ell^{2}(\mathbb{N})).
$$
Then  $S\phi = \psi S$ means that
$$
\left(
\begin{array}{cccc}
0&0&0&\cdots\\
c_{11}&c_{12}&c_{13}&\cdots\\
c_{21}&c_{22}&c_{23}&\cdots\\
\vdots&\ddots&\ddots&\ddots\\
\end{array}\right)
= 
\left(
\begin{array}{cccc}
\beta_{1}&\beta_{2}&\beta_{3}&\cdots \\ 
d_{11}&d_{12}&d_{13}&\cdots\\
d_{21}&d_{22}&d_{23}&\cdots\\
\vdots&\ddots&\ddots&\ddots
\end{array}\right),
$$
Hence we have $(c_{ij})_{ij} = (d_{ij})_{ij}$ and 
$$
\left(
\begin{array}{cccc}
\beta_{1}&\beta_{2}&\beta_{3}&\cdots \\ 
\end{array}\right)
=\left(
\begin{array}{cccc}
0&0&0&\cdots \\ 
\end{array}\right). 
$$
Therefore we have 
$$
\phi=\left(
\begin{array}{ccc}
c_{11}&c_{12}&...\\
c_{21}&c_{22}&...\\
\vdots&...&...
\end{array}\right), 
\ \  \text{ and } \ \ 
\psi=\left(
\begin{array}{cccc}
\alpha&0&0&... \\ 
\gamma_{1}&c_{11}&c_{12}&...\\
\gamma_{2}&c_{21}&c_{22}&...\\
\vdots&...&...&...
\end{array}\right)
$$
Next we consider another compatibility condition $T\phi=\psi T$. 
Put $\widetilde{c_{k}}=\sum_{i=1}^{\infty}w_{i}c_{ik}$. 
Then we have 
\begin{align*}
T\phi
&=\left(
\begin{array}{cccc}
w_{1}&w_{2}&w_{3}&\cdots \\ 
\lambda_{1}&0&0&\cdots\\
0&\lambda_{2}&0&\cdots\\
\vdots&\ddots&\ddots&\ddots
\end{array}\right)
\left(
\begin{array}{cccc}
c_{11}&c_{12}&c_{13}&\cdots\\
c_{21}&c_{22}&c_{23}&\cdots\\
c_{31}&c_{32}&c_{33}&\cdots\\
\vdots&\ddots&\ddots&\ddots\\
\end{array}\right)  \\
&=\left(
\begin{array}{cccc}
\widetilde{c_{1}}&\widetilde{c_{2}}&\widetilde{c_{3}}&\cdots\\
\lambda_{1}c_{11}&\lambda_{1}c_{12}&\lambda_{1}c_{13}&\cdots\\
\lambda_{2}c_{21}&\lambda_{2}c_{22}&\lambda_{2}c_{23}&\cdots\\
\vdots&\ddots&\ddots&\ddots\\
\end{array}\right)
\end{align*}
We also have that 
\begin{align*}
\psi T
& =
\left(
\begin{array}{cccc}
\alpha&0&0&\cdots \\ 
\gamma_{1}&c_{11}&c_{12}&\cdots\\
\gamma_{2}&c_{21}&c_{22}&\cdots\\
\vdots&\ddots&\ddots&\ddots
\end{array}\right)
\left(
\begin{array}{cccc}
w_{1}&w_{2}&w_{3}&\cdots \\ 
\lambda_{1}&0&0&\cdots\\
0&\lambda_{2}&0&\cdots\\
\vdots&\ddots&\ddots&\ddots
\end{array}\right) \\
& =
\left(
\begin{array}{cccc}
\alpha w_{1}&\alpha w_{2}&\alpha w_{3}&\cdots \\ 
\gamma_{1}w_{1}+c_{11}\lambda_{1}
&\gamma_{1}w_{2}+c_{12}\lambda_{2}
&\gamma_{1}w_{3}+c_{13}\lambda_{3}
\cdots\\
%%%%%%%%%%%%%%%%%%%%%%%%%%%%%%%%%%%%%%%%%%
\gamma_{2}w_{1}+c_{21}\lambda_{1}
&\gamma_{2}w_{2}+c_{22}\lambda_{2}
&\gamma_{2}w_{3}+c_{23}\lambda_{3}
\cdots\\
%%%%%%%%%%%%%%%%%%%%%%%%%%%%%%
\vdots&\ddots&\ddots&\ddots
\end{array}\right)
\end{align*}

Hence we have the following relations: 
$\alpha w_{i}=\widetilde{c_{i}}(i=1,2,3.\cdots)$ 
and 
$$
\gamma_{1}w_{1}+c_{11}\lambda_{1}=\lambda_{1}c_{11},
$$
$$
\gamma_{1}w_{2}+c_{12}\lambda_{2}=\lambda_{1}c_{12},
$$
$$
\gamma_{1}w_{3}+c_{13}\lambda_{3}=\lambda_{1}c_{13},
\cdots.
$$
Since $\gamma_{1}w_{1}+c_{11}\lambda_{1}=\lambda_{1}c_{11}$,
we have $\gamma_{1}w_{1}=0$.
$w_{1}\ne 0$ implies that $\gamma_{1}=0$.
Therefore
$c_{12}\lambda_{2}=\lambda_{1}c_{12}$. 
Since $\lambda_{1}\ne \lambda_{2}$, 
we have $c_{12}=0$. 
Similarly we have $c_{1j}=0$ for $j\ne 1$.

Next look at the part including $\gamma_{2}$ parameter:
$$
\gamma_{2}w_{1}+c_{21}\lambda_{1}=\lambda_{2}c_{21},
\gamma_{2}w_{2}+c_{22}\lambda_{2}=\lambda_{2}c_{22},
\gamma_{2}w_{3}+c_{23}\lambda_{3}=\lambda_{2}c_{23},
\cdots.
$$
Since $\gamma_{2}w_{2}=0$ and 
$w_{2}\ne 0$, we have  $\gamma_{2}=0$.
By a similar  argument as above,
 we have $c_{2j}=0$ for $j\ne 2$.
In the same way, we have $\gamma_{i}=0$ 
for any $i$ and 
$c_{ij}=0$ for $i\ne j$.
Therefore we have that 
$$
\psi=
\left(
\begin{array}{cccc}
\alpha&0&0&\cdots \\ 
0&c_{11}&0&\cdots\\
0&0&c_{22}&\cdots\\
\vdots&\ddots&\ddots&\ddots
\end{array}\right), 
\ \ \text{ and } \ \ 
\phi
=\left(
\begin{array}{cccc}
c_{11}&0&0&\cdots \\ 
0&c_{22}&0&\cdots\\
0&0&c_{33}&\cdots\\
\vdots&\ddots&\ddots&\ddots
\end{array}\right).
$$  

Since 
$\widetilde{c_{1}}=\alpha w_{1}$ 
and 
$\widetilde{c_{1}}=
w_{1}c_{11}+w_{2}c_{21}+w_{3}c_{31}+\cdots = w_{1}c_{11}$,
we have $\alpha=c_{11}$, because $w_{1}\ne 0$. 
Similarly the equations 
$\widetilde{c_{2}}=\alpha w_{2}$,
$\widetilde{c_{3}}=\alpha w_{3},\cdots$  
and $w_{i}\ne 0$, we have 
$\alpha= c_{ii}$ for any $i$. 
Hence $(\phi,\psi) = (\alpha I,\alpha I). $
Thus $(H,f)$ is transitive. 
\end{proof}

Next we shall construct a transitive Hilbert representation of 
the Kronecker quiver in another method. It is 
a modification of an unbounded operator 
which provides a transitive lattice 
by Harrison,Radjavi and Rosenthal \cite{HRR}.

Let $A$  and $B$ be bounded operators on a 
Hilbert space ${\mathcal H} $. Let 
$(H^{(A,B)},f^{(A,B)})$ be a Hilbert representation 
of the Kronecker quiver $K$
defined by $H_1 = {\mathcal H}, H_2 = {\mathcal H}, 
f_{\alpha} = A$ and $f_{\beta}= B$. 
If $A$ is invertible, then 
$(H^{(A,B)},f^{(A,B)})$ is isomorphic to 
$(H^{(I,A^{-1}B)},f^{(I,A^{-1}B)})$. 
Even if $A$ does not have a bounded inverse, 
$A^{-1}B$ can be  an unbounded operator. Hence 
if an unbounded operator $C$ is formally written 
by $C = "A^{-1}B"$, then we might replace 
$(H^{(I,C)},f^{(I,C)})$ by 
$(H^{(A,B)},f^{(A,B)})$ to keep it in the category of 
bounded operators. We shall adapt the idea to  
an unbounded operator $C$ which 
gives a transitive lattice in  \cite{HRR}. 
We also modify it a little bit to make a calculation of 
$End(H,f)$ easier.

\begin{thm}
Let $K$ be the  Kronecker quiver and 
${\mathcal H} = \ell^{2}(\mathbb{Z})$. 
Fix a positive constant  $\lambda>1$. 
Consider two weight sequences 
$a = (a(n))_{n \in \mathbb{Z}}$ and 
$b = (b(n))_{n \in \mathbb{Z}}$ by 
$$
a(n)= 
\begin{cases} e^{-\lambda^{n}},&   (n\geq 1, n \text{\ is even }) \\
  1, &  \ \  (otherwise), 
\end{cases}
\ \ \ \ \ 
b(n)=
\begin{cases} e^{-\lambda^{n}},&   (n\geq 1,n \text{\ is odd })  \\
 1, & \ \  (otherwise).
\end{cases}
$$
Let $D_{a}$  be a diagonal operator with 
$a =(a(n))_n$ as diagonal coefficients. 
and 
$D_{b}$ be a diagonal operator with 
$b =(b(n))_n$ as diagonal coefficients. 
Let $U$ be the bilateral  forward shift. 
Put $A = D_{a}$ and $B=UD_{b}$, 
so that
$A$ is a positive operator and 
$B$ is a weighted  forward shift operator.
Define a Hilbert representation $(H^{\lambda},f^{\lambda})$ of 
the Kronecker quiver $K$ by 
$H^{\lambda}_{0}={\mathcal H},H^{\lambda}_{1}={\mathcal H}$, 
$f^{\lambda}_{\alpha}=A$ and $f^{\lambda}_{\beta}=B$. 
Then the Hilbert representation $(H^{\lambda},f^{\lambda})$ 
of $K$ is transitive and is not isomorphic to any of 
the Hilbert representation in Theorem \ref{thm:perturbation}
 constructed by 
a perturbation of a weighted shift  by a rank-one operator. 
\end{thm}
\begin{proof}
Let $\{e_n\}_{n \in {\mathbb Z}}$ be the canonical basis of 
${\mathcal H} = \ell^{2}(\mathbb{Z})$.

Let $T=(T_{1},T_{2})$ be in $End(H^{\lambda},f^{\lambda}).$
Then  $T_{2}A=AT_{1}$ and  $T_{2}B=BT_{1}$, that is, 
$T_{2}D_{a}=D_{a}T_{1}$ and  $T_{2}UD_{b}=UD_{b}T_{1}.$
Since
$(T_{2}D_{a}e_{n}\mid e_{m})=(D_{a}T_{1}e_{n}\mid e_{m}),$
we have 
$a(n)(T_{2}e_{n}\mid e_{m})=a(m)(T_{1}e_{n}\mid e_{m}).$
Hence
$$
(T_{1}e_{n}\mid e_{m})=\frac{a(n)}{a(m)}(T_{2}e_{n}\mid e_{m}).
$$
Since $(T_{2}UD_{b}e_{n}\mid e_{m})=(UD_{b}T_{1}e_{n}\mid e_{m}),$
we have
$b(n)(T_{2}e_{n+1}\mid e_{m})=(T_{1}e_{n}\mid {b(m-1)}e_{m-1}).$
Replacing $m$ by $m+1$, we obtain 
$$
(T_{1}e_{n}\mid e_{m})= \frac{b(n)}{b(m)}(T_{2}e_{n+1}\mid e_{m+1}).
$$
Combining these equations, we have 
$\frac{a(n)}{a(m)}(T_{2}e_{n}\mid e_{m})
=\frac{b(n)}{b(m)}(T_{2}e_{n+1}\mid e_{m+1}).$
Put  $w_{m}=\frac{b(m)}{a(m)}. $ 
Then  $w_{m} =e^{(-\lambda)^{m}}$ for $m\geq 1$. 
$$ 
(T_{2}e_{n+1}\mid e_{m+1})
=\frac{w_{m}}{w_{n}}(T_{2}e_{n}\mid e_{m}).
$$
Since  $m-n$ component $t_{m,n}$ of $T_2$ is given by  
$t_{m,n} = (T_{2}e_{n}\mid e_{m})$, 
$t_{m+1,n+1} = \frac{w_{m}}{w_{n}}t_{m,n}$. 
Putting $m =n$, we have $t_{n+1,n+1} = t_{n,n}$.
This shows that the diagonal of $T_2$ is a constant. 
We shall show that $T_2$ is a scalar operator. 
Suppose that  $T_{2}$ were not a scalar operator. 
Then there exist integers $m\neq n$ such that $t_{m,n}\neq 0$.  
For any integer $k \geq 1$, we have 
$$
t_{m+k,n+k}
= \frac{w_{m}w_{m+1}\cdots w_{m+k-1}}{w_{n}w_{n+1}\cdots w_{n+k-1}}
t_{m,n}. 
$$
Define $c_{k}(m,n)$ by 
$\frac{w_{m}w_{m+1}\cdots w_{m+k-1}}{w_{n}w_{n+1}\cdots w_{n+k-1}}
=e^{c_{k}(m,n)}.$  First consider the case that $m \geq 1$ and 
$n \geq 1$. 
Then we have 
\begin{align*}
c_{k}(m,n) 
& =((-\lambda)^{m}+(-\lambda)^{m+1}+\cdots +(-\lambda)^{m+k-1}) \\
& -((-\lambda)^{n}+(-\lambda)^{n+1}+\cdots +(-\lambda)^{n+k-1})  \\
& =\frac{(-\lambda)^{m}(1-(-\lambda)^{k})}{1+\lambda}
-\frac{(-\lambda)^{n}(1-(-\lambda)^{k})}{1+\lambda}.
\end{align*}
Since 
$\limsup_{k} c_{k}(m,n)=\infty$, we have  
$\limsup_{k}t_{m+k,n+k} = \infty.$ 
This contradicts to the boundedness of $T_{2}$. 
The other cases are similarly proved. 
Consequently $T_{2} = \alpha I$ for some constant $\alpha$. 

Since 
$(T_{1}e_{n}\mid e_{m})=\frac{a(n)}{a(m)}(T_{2}e_{n}\mid e_{m}),$
we have $(T_{1}e_{n}\mid e_{m}) = 0$ for $m \not= n$ and 
$(T_{1}e_{n}\mid e_{n}) = \alpha$ for any $n \in {\mathbb Z}$. 
Hence 
$T_{1}=T_{2}= \alpha I$.
This shows that the Hilbert representation 
$(H^{\lambda},f^{\lambda})$ is transitive.

The Hilbert representation $(H^{\lambda},f^{\lambda})$ 
is not isomorphic to any of 
the Hilbert representation $(H,f)$ in \ref{thm:perturbation} constructed by 
a perturbation of a weighted shift  by a rank-one operator. In fact 
the image of $f_{\alpha}$ is closed but  the 
image of $f^{\lambda}_{\alpha}$ is not closed. 

\end{proof}

A little more careful calculation shows that 
$(H^{\lambda},f^{\lambda})$ and $(H^{\mu},f^{\mu})$ 
are not isomorphic if $\lambda \not= \mu $, 
$\lambda > 1$  and $\mu > 1$ as follows: 

\begin{thm}
Let $K$ be the Kronecker quiver. 
Then there exist continuously many non-isomorphic 
Hilbert representations $(H^{\lambda},f^{\lambda})$ 
of the  Kronecker quiver $K$ which are transitive. 
\end{thm}
\begin{proof}
Let $\lambda$ and $\mu$ be positive constants such that 
$\lambda \not= \mu, \ 
\lambda > 1$  and $\mu > 1$. It is enough to show that 
$Hom ((H^{\lambda},f^{\lambda}),(H^{\mu},f^{\mu})) = 0$.  
We shall write 
$A^{\mu},B^{\mu},{a}^{\mu}$ and ${b}^{\mu}$ for 
$(H^{\mu},f^{\mu})$. 
Take any  homomorphism 
$(T_{1},T_{2}) \in Hom((H^{\lambda},f^{\lambda}),(H^{\mu},f^{\mu}))$. 
Then  we have 
$T_{2}A^{\lambda}=A^{\mu}T_{1}$ and 
$T_{2}B^{\lambda}=B^{\mu}T_{1}$.
Since $(T_{2}A^{\lambda}e_{n}\mid e_{m})=(A^{\mu}T_{1}e_{n}\mid e_{m})$, 
we have 
$a^{\lambda}(n)(T_{2}e_{n}\mid e_{m})=(T_{1}e_{n}\mid a^{\mu}(m)e_{m})$.
Hence we obtain that 
$$
(T_{1}e_{n}\mid e_{m})=\frac{a^{\lambda}(n)}{a^{\mu}(m)}(T_{2}e_{n}\mid e_{m})
$$
Next consider 
$(T_{2}B^{\lambda}e_{n}\mid e_{m})=(B^{\mu}T_{1}e_{n}\mid e_{m})$.
Since 
$(T_{2}UD_{b}^{\lambda}e_{n}\mid e_{m}) 
=(T_{1}e_{n}\mid D_{b}^{\mu}U^{*}e_{m})$, 
we have that 
${b}^{\lambda}(n)(T_{2}e_{n+1}\mid e_{m}) 
= {b}^{\mu}(m-1)(T_{1}e_{n}\mid e_{m-1})$.  
Replacing $m$ by $m+1$, we obtain 
$$
(T_{1}e_{n}\mid e_{m})
=\frac{{b}^{\lambda}(n)}{{b}^{\mu}(m)}(T_{2}e_{n+1}\mid e_{m+1}). 
$$
Combining these equations, we have 
$$(T_{2}e_{n+1}\mid e_{m+1})
=\frac{{a}^{\lambda}(n)}{{a}^{\mu}(m)}\frac{b^{\mu}(m)}{b^{\lambda}(n)}
(T_{2}e_{n}\mid e_{m})
$$.

We put $w^{\mu}_{m}=\frac{b^{\mu}(m)}{{a}^{\mu}(m)}$. 
Then 
$$
(T_{2}e_{n+1}\mid e_{m+1})
= \frac{w^{\mu}_{m}}{w^{\lambda}_{n}}(T_{2}e_{n}\mid e_{m}).
$$
For any integer $k \geq 1$, we have 

$$ 
(T_{2}e_{n+k}\mid e_{m+k})
= \frac{w^{\mu}_{m}\cdots w^{\mu}_{m+k-1}}
{w^{\lambda}_{n}\cdots w^{\lambda}_{n+k-1}}(T_{2}e_{n}\mid e_{m})
$$
Define $c_{k}(m,n)$ by 
$$
 \frac{w^{\mu}_{m}\cdots w^{\mu}_{m+k-1}}
{w^{\lambda}_{n}\cdots w^{\lambda}_{n+k-1}}
= e^{c_{k}(m,n)}.
$$
Then we have 
$$
c_{k}(m,n)
=((-\mu)^{m}+\cdots (-\mu)^{m+k-1})-
((-\lambda)^{n}+\cdots (-\lambda)^{n+k-1})
$$
Since $\mu \neq \lambda$, we may and do assume that $1<\lambda<\mu$.
We shall show that $T_{2} =0$.  On the contrary, assume that 
$T_{2}\neq 0$. Then there exist integers 
$m,n$ such that $(T_{2}e_{n}\mid e_{m}) \not= 0 $. 
We can  show that $\limsup_{k}c_{k}(m,n)=\infty.$ 
For example, if $m \geq 1$ and $n \geq 1$, then 
$$
c_{k}(m,n)=\frac{(-\mu)^{m}(1-(-\mu)^{k})}{1+\mu}
\{1-\frac{(-\lambda)^{n}(1+\mu)(1-(-\lambda)^{k})}
{(-\mu)^{m}(1+\lambda)(1-(-\mu)^{k})}\}
$$
Hence $\limsup_{k}c_{k}(m,n)=\infty.$ The rest cases are similarly 
proved. 
But this contradicts that $T_{2}$ is bounded.
Therefore  $T_{2}=0$.
Since 
$(T_{1}e_{n}\mid e_{m})=\frac{a^{\lambda}(n)}{a^{\mu}(m)}(T_{2}e_{n}\mid e_{m})$,
we also have $T_{1}$=0.
This shows that $T=(T_{1},T_{2})=0$.
Hence we have that 
$Hom((H^{\lambda},f^{\lambda}),(H^{\mu},f^{\mu}))=0$.
Therefore 
$(H^{\lambda},f^{\lambda})$ is not isomorphic to $(H^{\mu},f^{\mu})$.
\end{proof}

\section{Difference between purely algebraic version and Hilbert space 
version}

There exist subtle difference among  purely algebraic infinite-dimensional 
representations of quivers, infinite-dimensional Banach (space) 
representations of quivers and infinite-dimensional Hilbert (space)
representations of quivers.  We also note that the analytic aspect of 
Hardy space is quite important in our setting.

We shall use the following elementary fact: Let $A$ be a 
unital algebra and $L(A)$ be the set of linear operators 
on $A$. . Let $\lambda : A \rightarrow L(A)$ be the 
left multiplication such that $\lambda_ax = ax$ for $a,x\in A$. 
Similarly let  $\rho : A \rightarrow L(A)$ be the  
right multiplication such that $\rho_ax = xa$ for $a,x \in A$. 
Then the commutant 
$$
\lambda(A)' := \{T \in L(A) \ | \ T\lambda_a = \lambda_a T 
 \text{ for any } a \in A \}
$$ 
is exactly $\rho(A)$. In fact, let $T \in \lambda(A)'$.  
 Put $b = T(1) \in A$. Then $T\lambda_a (1) = \lambda_a T(1)$ for any 
$a \in A$. 
Hence $T(a) = ab = \rho_b(a)$. Therefore $T = \rho_b$.

\noindent  
{\bf Definition.}
Let $\Gamma=(V,E,s,r)$ be a finite quiver. We say
that $(K,f)$ is a  {\it Banach representation} of $\Gamma$ 
if $K=(K_{v})_{v\in V}$  is a family of  Banach spaces 
and $f=(f_{\alpha})_{\alpha \in E}$ is a family of
 bounded linear operators such that $f_{\alpha} : 
K_{s(\alpha)}\rightarrow K_{r(\alpha)}$ for $\alpha \in E$. 
A Banach  representation $(K,f)$ of $\Gamma$ 
is called  {\it decomposable} if 
$(K,f)$ is isomorphic to a direct sum of two 
non-zero Banach representations.  
A non-zero Banach representation $(H,f)$ of $\Gamma$ 
is said to be  {\it indecomposable} if 
it is not decomposable. The other notions are similarly defined 
for Banach representations. A Banach  representation $(H,f)$ is 
indecomposable if and only if any endomorphism of $(H,f)$ which 
is 
idempotent is $0$ or $I$. 

\bigskip
\noindent
{\bf Example 10.}Let $L_1$ be one-loop quiver, that is, 
$L_1$ is a quiver with one vertex $\{1\}$ and 
$1$-loop $\{\alpha \}$ such that $s(\alpha) = r(\alpha) = 1$. 
 Consider a purely algebraic 
group algebra $V_1 := {\mathbb C}[{\mathbb Z}]$, the 
reduced group $C^*$-algebra 
$K_1:=C_r^*({\mathbb Z})$, the group von Neumann algebra $W^*({\mathbb Z})$ 
and a Hilbert space $H_1 =\ell^2({\mathbb Z})$. 
We identify $V_1$ with the algebra of finite Laurant polynomials. 
As a set we have inclusions under a certain identification:
$$
 V_1 = {\mathbb C}[{\mathbb Z}] \subset K_1 = C_r^*({\mathbb Z}) 
\subset W^*({\mathbb Z}) \subset \ell^2({\mathbb Z}) = H_1 . 
$$
Moreover $V_1$ is a dense subset of $H_1$ with respect to 
 the topology of $\ell^2$ norm of 
$H_1$. Define a purely algebraic representation $(V,T)$ of $L_1$ 
 by $V_1 = {\mathbb C}[{\mathbb Z}]$ and the 
multiplication operator $T_{\alpha}$ by $z$.  That is, 
$T_{\alpha}h(z) = zh(z)$ for a finite Laurant polynomial 
$h(z) = \sum_n a_n z^n \in {\mathbb C}[{\mathbb Z}]$. 
Similarly we can define a Banach space representation 
$(K,S)$ of $L_1$ 
 by $K_1 =  C_r^*({\mathbb Z}) 
\cong C({\mathbb T})$ and the 
multiplication operator $S_{\alpha}$ by $z$. 

Since 
$$
End(V,T) \cong {\mathbb C}[{\mathbb Z}] \subset  C_r^*({\mathbb Z}) 
\cong C({\mathbb T})
$$ 
and $C({\mathbb T})$  
have no non-trivial idempotents, 
the purely algebraic representation $(V,T)$ and 
the Banach  representation $(K,S)$ is indecomposable. 
On the other hand, the multiplication operator $U_{\alpha}$ by $z$ 
on $H_1 = \ell^2({\mathbb Z})\cong L^2({\mathbb T})$ 
gives a Hilbert representation $(H,U)$. Since $U$ is a unitary,  
 for any operator $A$, $AU_{\alpha} = U_{\alpha}A$ 
implies $AU_{\alpha}^* = U_{\alpha}^*A$.  
Therefore 
$$
End(H,U)  \cong \{A \in B(\ell^2({\mathbb Z})) 
\ |\ AU_{\alpha} = U_{\alpha}A \} 
\cong L^{\infty}({\mathbb T}) . 
$$ 
because $L^{\infty}({\mathbb T})$ is maximal abelian in 
$B(L^2({\mathbb T})) \cong B(\ell^2({\mathbb Z}))$. 
Since $L^{\infty}({\mathbb T})$ has many non-trivial idempotents, 
the Hilbert representation $(H,U)$ 
is {\it not} indecomposable. Therefore the  completion by the 
$L^2$-norm  changes the indecomposability but the 
the  completion by the 
sup-norm  does not change the indecomposability. 
The example suggests  
that proving  indecomposability 
for Hilbert representations is sometimes more difficult 
than proving indecomposability for purely algebraic representations.

We shall extend the above example to the $n$-loop quiver. 

\begin{prop}
Let $L_n$ be the $n$-loop quiver, 
that is, 
$L_n$ is a quiver with one vertex $\{1\}$ and 
$n$-loops $\{\alpha_k \ | k =1, \dots, n \}$ 
with $s(\alpha_k) = r(\alpha_k) = 1$ for $k = 1,\dots, n$. 
Let $F_n$ be the (non-abelian) free group of $n$-generators 
$\{a_1, \dots, a_n\}$. 
Consider the  purely algebraic 
group algebra $V_1 := {\mathbb C}[F_n]$, 
the reduced group $C^*$-algebra 
$K_1 =: C^*_r(F_n)$, 
the group von Neumann algebra $W^*(F_n)$ 
and a Hilbert space $H_1 =\ell^2(F_n)$. 
Let $\{\delta_g \ | \ g \in F_n\}$ be a standard basis of 
$H_1 =\ell^2(F_n)$ such that $\delta_g(h) = 1$ if $g=h$ 
and $\delta_g(h) = 0$ if $g \not=h$. 
Let $\lambda: F_n \rightarrow B(\ell^2(F_n))$ be the left regular
representation, that is, $\lambda_s(\delta_g) = \delta_{sg}$. 
The left regular representation defines a Hilbert representation 
$(H,f)$ by $H_1 =\ell^2(F_n)$ and $f_{\alpha_k} = \lambda_{a_k}$ 
for $k = 1,\dots, n$. Define a Banach representation $(K,U)$ by 
its restriction to $C^*_r(F_n)$, that is, 
$K_1 =C^*_r(F_n)$ and $U_{\alpha_k} = \lambda_{a_k}|_{K_1}$ for
$k = 1,\dots, n$. 
We also consider a purely algebraic representation $(V,T)$ by representation
its restriction to ${\mathbb C}[F_n]$, that is, 
$V_1 = {\mathbb C}[F_n]$ and $T_{\alpha_k} = \lambda_{a_k}|_{V_1}$ 
for  $k = 1,\dots, n$. Then the  purely algebraic representation $(V,T)$
and Banach representation $(K,U)$ are indecomposable but 
the Hilbert representation $(H,f)$ is not indecomposable. 
\end{prop}
\begin{proof}
The reduced group $C^*$-algebra  
$C^*_r(F_n)$ is the $C^*$-algebra generated by 
$\{\lambda_s \ | \ s \in F_n\}$ and has no non-trivial idempotents. 
The  group von Neumann algebra $W^*(F_n)$ 
is the von Neumann algebra generated by 
$\{\lambda_s \ | \ s \in F_n\}$ and
has many non-trivial 
idempotents. The  purely algebraic 
group algebra ${\mathbb C}[F_n]$ 
is dense in $C^*_r(F_n)$ with respect to  operator norm 
and dense in $W^*(F_n)$ 
with respect to  weak operator topology. Moreover 
${\mathbb C}[F_n]$ is dense in $\ell^2(F_n)$ 
with respect to $\ell^2$ norm. Define an embedding 
$\eta: W^*(F_n) \rightarrow \ell^2(F_n)$ by 
$\eta(T) = T\delta_1 $. In this sense we have 
inclusions as  set: 
$$
{\mathbb C}[F_n] 
\subset C^*_r(F_n) \subset W^*(F_n) 
\subset \ell^2(F_n). 
$$

Since 
$$
End(H,f)  \cong \{A \in B(\ell^2({\mathbb Z})) \ | \ 
AB = BA \text{ for any }  B \in W^*(F_n )\}
$$ 
is the von Neumann algebra generated by the {\it right}
 regular representation 
and isomorphic to $W^*(F_n)$,  
$End(H,f)$ has many non-trivial idempotent. Hence 
the Hilbert representation $(H,f)$ 
is {\it not} indecomposable. 
On the other hand, 
$$
End(K,U) 
= C^*(\{ \rho_s \ | \  s \in F_n \})  \cong  C^*_r(F_n).
$$ 
Since $C^*_r(F_n)$ has no non-trivial projections, 
$C^*_r(F_n)$ has no non-trivial idempotents. 
Hence the Banach representation $(K,U)$ is indecomposable. 
Since
$$
End(V,T) \cong {\mathbb C}[F_n] \subset  C^*_r(F_n) 
$$ 
have no idempotents, 
the purely algebraic representation $(V,T)$ is also indecomposable. 
\end{proof} 

The same  phenomenon occurs for $(n+1)$-Kronecker quiver 
by a similar argument.

\begin{prop} 
Let $K_{n+1}$ be the $n+1$-Kronecker quiver 
, that is, 
$K_{n+1}$ is a quiver with two vertex $\{1, 2 \}$ and 
$n+1$ edges $\{\alpha_1, \dots, \alpha_{n+1} \}$ such that 
$s(\alpha_k) = 1$ and $r(\alpha_k) = 2$ for $k = 1,\dots, n+1$. 
Let $F_n$ be the (non-abelian) free group of $n$-generators 
$\{a_1, \dots, a_n\}$. 
Define a Hilbert representation $(H,f)$ of $K_{n+1}$ by 
$H_1 = H_2= \ell^2(F_n)$ and 
$f_{\alpha_k}= \lambda_{a_k}$ for $k =1, \dots, n$ and 
$f_{\alpha_{n+1}} = I$. 
We also define a Banach representation  $(K,U)$ and 
a purely algebraic representation $(V,T)$ of $K_{n+1}$ by the 
restriction to $ C^*_r(F_n)$ and ${\mathbb C}[F_n]$ respectively: 
Let  $K_1 = K_2 = C^*_r(F_n)$ and 
$U_{\alpha_k} =  \lambda_{a_k} |_{K_1}$ for $k = 1, \dots, n+1$. 
We also put $V_1 = V_2 = {\mathbb C}[F_n]$ and 
$T_{\alpha_k} =  \lambda_{a_k} |_{V_1}$ for $k = 1, \dots, n+1$. 
Then 
the  purely algebraic representation $(V,T)$ 
and Banach representation $(K,U)$ are indecomposable but 
the Hilbert representation $(H,f)$ is not indecomposable. 
\end{prop}
\begin{proof}
Since 
$End(H,f)  \cong \{(A,A) \in B(\ell^2({\mathbb Z}))^2 \ | \ 
AB = BA \text{ for any }  B \in W^*({F_n}) \}$ 
is isomorphic to $W^*({F_n})$ and 
$End(H,f)$ has many non-trivial idempotents, 
the Hilbert representation $(H,f)$ 
is {\it not} indecomposable. 
On the other hand, 
$$
End(K,U) \cong \{(A,A) \in (C_r^*[F_n])^2 \ | \ 
A \in  C_r^*[F_n] \}
$$
and 
$$
End(V,T) \cong \{(A,A) \in {\mathbb C}[F_n]^2 \ | \ 
A \in {\mathbb C}[F_n] \}, 
$$
have no non-trivial idempotents, $(K,U)$ and $(V,T)$ are indecomposable. 
\end{proof}

\bigskip
\noindent
{\bf Example 11. }Let $L_1$ be one-loop quiver, that is, 
$L_1$ is a quiver with one vertex $\{1\}$ and 
$1$-loop $\{\alpha \}$ such that $s(\alpha) = r(\alpha) = 1$. 
Consider two infinite-dimensional spaces the polynomial ring 
${\mathbb C}[z]$ and the Hardy space $H^2({\mathbb T})$. 
Then ${\mathbb C}[z]$ is dense in $H^2({\mathbb T})$ with 
respect to the Hilbert space norm. 

Define a purely algebraic representation $(V,T)$ of $L_1$ 
by $V_1 = {\mathbb C}[z]$ and the 
multiplication operator $T_{\alpha}$ by $z$.  That is, 
$T_{\alpha}h(z) = zh(z)$ for a polynomial $h(z) = \sum_n a_n z^n$. Since 
$End(V,T) \cong {\mathbb C}[z]$ 
have no idempotents, 
the purely algebraic representation $(V,T)$ is indecomposable.

Next we define a Hilbert representation $(H,S)$ by 
$H_1 = H^2({\mathbb T})$ and 
 $S_{\alpha}= T_z$ the Toeplitz operator with the symbol $z$. 
Then  $S_{\alpha}=T_z$ is the multiplication operator by $z$ 
on $H^2({\mathbb T})$ and is identified with the unilateral shift. 
Then 
\begin{align*}
End(H,S) & \cong \{A \in B(H^2({\mathbb T})) \ | \ AT_z = T_zA \} \\
& = \{T_{\phi} \in B(H^2({\mathbb T})) \ | \ 
\phi \in H^{\infty}({\mathbb T}) \}  
\end{align*}
is the algebra of analytic Toeplitz operators 
and isomorphic to $H^{\infty}({\mathbb T})$. 
By the F. and M. Riesz Theorem, if $f \in H^2({\mathbb T})$ has the 
zero set of positive measure, then $f = 0$. Since 
$H^{\infty}({\mathbb T}) = H^2({\mathbb T}) \cap L^{\infty}({\mathbb T})$, 
$H^{\infty}({\mathbb T})$ has no non-trivial idempotents. 
Thus  there exists no non-trivial idempotents which commutes 
with $T_z$ and  Hilbert space $(H,S)$ is indecomposable. 
In this sense, the analytical aspect of Hardy space is 
quite important in our setting.

Any subrepresentation of the purely algebraic representation $(V,T)$ 
is given by the restriction to 
an ideal $J=p(z){\mathbb C}[z]$ for some polynomial $p(z)$. 
Any subrepresentation of the Hilbert representation $(H,S)$ is given 
by an invariant subspace of the shift operator $T_z$. Beurling theorem 
shows that any subrepresentation of $(H,S)$ is given by the 
restriction to an invariant subspace $M = \varphi H^2({\mathbb T})$ for 
some inner function  $\varphi$. For example, if 
an ideal $J$ is defined  by 
$$
J = \{f(z) \in {\mathbb C}[z] \ | \ f(\lambda_1) = \dots = f(\lambda_n)= 0 \}
$$
for some distinct numbers $\lambda_1, \dots \lambda_n \in {\mathbb C}$, then 
the corresponding polynomial $p(z)$ is given by 
$p(z) = (z-\lambda_1) \dots (z-\lambda_n)$. The case of Hardy space 
is much more analytic. We shall identify $H^2({\mathbb T})$ with 
a subspace $H^2({\mathbb D})$ of analytic functions on the  open unit 
disc ${\mathbb D}$.  
If an invariant subspace $M$ is defined by 
$$
M = \{ f \in H^2({\mathbb D}) \ | \ \ f(\lambda_1) = \dots = f(\lambda_n)= 0 \}
$$
for some distinct numbers $\lambda_1, \dots \lambda_n \in {\mathbb D}$,
then the corresponding inner function  $\varphi$ is given by a 
finite Blaschke product 
$$
\varphi(z) = \frac{(z- \lambda_1)}{1-\overline{\lambda_1} z}
\dots \frac{(z- \lambda_n)}{1-\overline{\lambda_n} z}.
$$
Here we cannot use the notion of degree like 
polynomials and we must manage to treat 
orthogonality to find such an inner function  $\varphi$. 

\section{Hilbert representations and relative position of subspaces}

We studied  relative position of subspaces of a Hilbert space in 
\cite{EW1}, \cite{EW2}  and Hilbert representations of quivers 
in \cite{EW3}. In this section we shall describe 
a relation between them, which is similar to purely algebraic 
situation and easy but  quite suggestive. Therefore we shall 
describe it here.

Let $H$ be a Hilbert space and $E_1, \dots E_n$ be $n$-subspaces 
in $H$.  Then we say that  ${\mathcal S} = (H;E_1, \dots , E_n)$  
is a system of $n$-subspaces in $H$.  
Let ${\mathcal T} = (K;F_1, \dots , F_n)$  
be  another system of $n$-subspaces in a Hilbert space $K$. Then  
$\varphi : {\mathcal S} \rightarrow {\mathcal T}$ is called a 
homomorphism if $\varphi : H \rightarrow K$ is a bounded linear 
operator satisfying that  
$\varphi(E_i) \subset F_i$ for $i = 1,\dots ,n$. And 
$\varphi : {\mathcal S} \rightarrow {\mathcal T}$
is called an isomorphism if $\varphi : H \rightarrow K$ is 
an invertible (i.e., bounded  bijective) linear 
operator satisfying that  
$\varphi(E_i) = F_i$ for $i = 1,\dots ,n$. 
We say that systems ${\mathcal S}$ and ${\mathcal T}$ are 
{\it isomorphic} if there is an isomorphism  
$\varphi : {\mathcal S} \rightarrow {\mathcal T}$. This means 
that the relative positions of $n$ subspaces $(E_1, \dots , E_n)$ in $H$ 
and   $(F_1, \dots , F_n)$ in $K$ are same under disregarding angles. 
We denote by 
$Hom(\mathcal S, \mathcal T)$ the set of homomorphisms of 
$\mathcal S$ to $\mathcal T$ and  
$End(\mathcal S) := Hom(\mathcal S, \mathcal S)$ 
the set of endomorphisms on $\mathcal S$. 
 A system $\mathcal S=(H;E_{1},\dots,E_{n})$
of $n$ subspaces is called {\it decomposable} 
if the system $\mathcal S$ is isomorphic to 
a direct sum of two non-zero systems.  
A non-zero system $\mathcal S=(H;E_{1},\cdots,E_{n})$ 
of $n$-subspaces is said to be 
{\it indecomposable} if it is not decomposable.

We recall that strongly irreducible operators $A$ play a 
crucial role to construct indecomposable systems of 
four subspaces. Moreover the commutant $\{A\}'$ corresponds to 
the endomorphism ring. 

\begin{thm}[\cite{EW1}]
For any single operator $A \in B(K)$ on a Hilbert space $K$, let
$\mathcal{S}_A = (H;E_1,E_2,E_3,E_4)$
be the associated  operator system such that 
$H = K \oplus K$ and 
\[
E_{1}=K\oplus 0,
E_{2}=0\oplus K, 
E_{3}=\{(x,Ax); x\in K\},
E_{4}=\{(y,y); y\in K\}. 
\]
Then
$$
End (\mathcal{S}_{A}) = \{ T \oplus T \in B(H) ; 
T \in B(K), \ AT = TA \}
$$
is isomorphic to the commutant $\{A\}$. 
The associated system $\mathcal{S}_A$ of four subspaces 
is indecomposable if and only if $A$ is strongly irreducible. 
Moreover for any operators $A,B  \in B(K)$ on a Hilbert space $K$, 
the associated systems $\mathcal{S}_A$ and $\mathcal{S}_B$ 
are isomorphic if and only if $A$ and $B$ are similar. 
\end{thm}

\bigskip
\noindent
{\bf Example 12. } We shall apply the above theorem 
to the famous facts on weighted shift operators and 
analytic function theory, see A. Shields \cite{Sh}. 
Let $a = (a_n)_{n \in \mathbb {N}}$ be 
a bounded sequence and $W_a$ be the associated weighted 
unilateral shift. Then the associated system 
$\mathcal{S}_{W_a}$ of four subspaces is indecomposable 
if and only if $W_a$ is strongly irreducible 
if and only if $a_n \not= 0$ for any $n \in \mathbb {N}$. 
Let $b = (b_n)_{n \in \mathbb {N}}$ be another 
bounded sequence. Then the associated system 
$\mathcal{S}_{W_a}$ and $\mathcal{S}_{W_b}$ are isomorphic 
if and only if $W_a$ and $W_b$  are similar 
if and only if there exist positive constants $C_1$ and 
$C_2$ such that for any $n \in \mathbb {N}$ 
$$
0  < C_1 \leq \frac{|a_1 \dots a_n|}{|b_1 \dots b_n|}\leq C_2 .
$$
Moreover $End (\mathcal{S}_{W_a})$ is isomorphic to the 
commutant $\{W_a\}'$, which is isomorphic to the \lq\lq analytic"
algebra $H^{\infty}(\beta)$ in the sense of \cite{Sh}, that is, 
the  class of formal power series $\phi$ such that 
$\phi H^2(\beta) \subset H^2(\beta)$, where 
${\beta}_0 = 1$, ${\beta}_n = a_0a_1 \dots a_{n-1}$ and 
$$
H^2(\beta) = \{ f | f(z) = \sum_{n=0}^{\infty} c_nz^n, \ 
 \sum _{n=0}^{\infty} |c_n{\beta}_n|^2 < \infty \}.  
$$

It is easy but fundamental to see that 
the study of relative positions of subspaces is reduced 
to the study of Hilbert representations of quivers. In 
particular the indecomposabilty is preserved:  
Let ${\mathcal S} = (H;E_1, \dots , E_n)$  
be a system of $n$-subspaces in a Hilbert space $H$.
Let $R_n=(V,E,s,r)$ be a $n$ subspace quiver 
such that $V = \{1,2,\dots,n,n+1 \}$ and 
$E = \{\alpha_k \ | \ k = 1, \dots, n \}$ with $s(\alpha_k) = k$ 
and $r(\alpha_k) = n+1$ for $k = 1, \dots, n$. 
Then there exists a Hilbert representation $(K,f)$ of $R_n$ such that 
$K_k = E_k$, $K_{n+1} = H$ and $f_{\alpha_k} : E_k \rightarrow H$ is an 
inclusion for  $k = 1, \dots, n$. Then there exists an algebra isomorphism 
$$
\theta:  End(\mathcal S) \rightarrow End(K,f)
$$
such that $\theta(\varphi) = (\varphi|{K_k})_{k \in V}$ for 
$\varphi \in End(\mathcal S)$. Therefore it is clear that the system 
 $\mathcal S$ of $n$ subspaces is indecomposable 
(resp. transitive) if and only if the corresponding Hilbert representation 
$(K,f)$ is indecomposable (resp. transitive).

We shall show a converse in a sense as same as  the purely algebraic 
version. 

\begin{lemma}
Let $\Gamma=(V,E,s,r)$ be a finite quiver without self-loops 
such that $V = \{v_1,\dots, v_n\}$ and $E = \{\alpha_1, \dots, \alpha_m\}$.
Let $(K,f)$ be a Hilbert representation of $\Gamma$. Then 
there exists a system ${\mathcal S} = (H;E_1, \dots , E_{n+m})$
 of $n+m$-subspaces such that 
 $End(K,f)  \cong End(\mathcal S)$. 
\label{lemma:reduction to subspaces}
\end{lemma} 
\begin{proof}
Let $H = K_{v_1} \oplus K_{v_2} \oplus \dots \oplus K_{v_n}$. 
Define $E_i = 0 \oplus \dots \oplus K_{v_i} \oplus \dots \oplus 0$ 
for $i = 1,2, \dots , n$. For $j = 1,\dots, m$, consider 
$$
f_{\alpha _{j}} : K_{s(\alpha_j)} \rightarrow K_{r(\alpha_j)}. 
$$
Define 
\begin{align*}
& E_{n + j} = "graph \ f_{\alpha _j}" 
:= \{x=(x_k)_k \in H \ | \\
&x = (0,\dots,0, z, \dots , f_{\alpha _{j}}(z),0 \dots,0) 
 \ z \in K_{s(\alpha_j)}, f_{\alpha _{j}}(z) \in K_{r(\alpha_j)} \}. 
\end{align*}
Then ${\mathcal S} := (H;E_1, \dots , E_{n+m})$ is a system 
of $n+m$-subspaces. 

For $T =(T_{v_k})_k \in End(K,f)$, define a bounded operator 
$S: H \rightarrow H$ by $S = diagonal (T_{v_1},\dots, T_{v_n})$. 
Then it is clear that $S(E_i) \subset E_i$ for $i = 1, \dots, n$. 
Since $T_{r(\alpha_j)}f_{\alpha_j} = f_{\alpha_j}T_{s(\alpha_j)}$, 
$S(E_{n+j}) \subset E_{n+j}$ for $j = 1, \dots m$. 
Hence $\varphi(T) = S$ define a homomorphism 
$\varphi: End(K,f) \rightarrow End\  {\mathcal S}$. 
Conversely 
Let $S$ be in $End\  {\mathcal S}$. Since $S(E_i) \subset E_i$ 
for $i = 1, \dots, n$, we can define bounded operators 
$T_{v_i} : K_{v_i} \rightarrow K_{v_i}$ such that 
$S =  diagonal (T_{v_1},\dots, T_{v_n})$. Since 
$S(E_{n+j}) \subset E_{n+j}$ for $j = 1, \dots m$, we have 
$T_{r(\alpha_j)}f_{\alpha_j} = f_{\alpha_j}T_{s(\alpha_j)}$. 
This shows that $T :=(T_{v_k})_k \in End(K,f)$. Hence 
$\psi(S) = T$ define a homomorphism 
$\psi: End\  {\mathcal S} \rightarrow End(K,f)$.  
Since $\psi \circ \varphi = id$ and $ \varphi \circ \psi = id$, 
$End(K,f)  \cong End(\mathcal S)$. 
\end{proof}

\begin{thm}
Let $\Gamma=(V,E,s,r)$ be a finite quiver.  
Let $(K,f)$ be a Hilbert representation of $\Gamma$. Then 
there exists a natural number $n$ and 
a system ${\mathcal S}$ of $n$-subspaces such that  
 $End(K,f)  \cong End(\mathcal S)$. 
\end{thm} 
\begin{proof}
Let $\Gamma=(V,E,s,r)$ be a finite quiver and 
$(K,f)$ a Hilbert representation of $\Gamma$. 
Then there exists another finite quiver 
$\Gamma'=(V',E',s,r)$ without self-loops and 
a Hilbert representation $(K',f')$ of $\Gamma'$ 
such that $End(K,f) \cong End(K',f')$. 
In fact, replace "locally" each  n-loop 
$\alpha_1, \dots, \alpha_n : v \rightarrow v$ by $(n+1)$-
Kronecker-like edges  $\beta_1, \dots, \beta_{n+1}
 : v \rightarrow v'$ to get $\Gamma'$. Any edge 
$\gamma: v \rightarrow w (\not=v)$ is also  replaced by 
$\gamma ': v' \rightarrow w$. Then 

Define $(K',f')$ by $K'_v = K_v$ and $K'_{v'} = K_v$, 
and $f'_{\beta_1} = f_{\alpha_1}, \dots, f'_{\beta_n} = f_{\alpha_n}$ 
and $f'_{\beta_{n+1}} = id$. We also put 
$f'_{\gamma '}= f_{\gamma}$. 
Then we have  $End(K,f) \cong End(K',f')$.
Next apply the lemma \ref{lemma:reduction to subspaces} 
for the finite quiver 
$\Gamma'=(V',E',s,r)$ without self-loops and 
a Hilbert representation $(K',f')$ of $\Gamma'$.  
\end{proof}

\end{document}